\documentclass[11pt,a4paper]{article}

\usepackage[utf8]{inputenc}
\usepackage[T1]{fontenc}

\usepackage{amsmath}
\usepackage{amsfonts}
\usepackage{amssymb}
\usepackage{amsthm}

\usepackage{a4wide}
\usepackage{graphicx}
\usepackage{multirow}

\DeclareMathAlphabet{\mathdj}{U}{msb}{m}{n}
\newcommand{\E}{\ensuremath{\mathdj{E}}}
\newcommand{\Mn}{\ensuremath{\mathdj{M}}_n}
\newcommand{\N}{\ensuremath{\mathdj{N}}}
\newcommand{\Pg}{\ensuremath{\mathdj{P}}}
\newcommand{\R}{\ensuremath{\mathdj{R}}}
\newcommand{\Lp}{\ensuremath{\mathdj{L}}}
\newcommand{\Ltwo}{\ensuremath{\mathdj{L}}^2}

\newcommand{\argmin}{\mathop{\text{argmin}}}
\newcommand{\card}{\text{Card}}
\def\ind{{\mathchoice {\rm 1\mskip-4mu l} {\rm 1\mskip-4mu l} {\rm 1\mskip-4.5mu l} {\rm 1\mskip-5mu l}}}
\newcommand{\var}{\text{Var}}

\newcommand{\im}{\text{Im}}
\newcommand{\M}{\mathcal{M}}
\newcommand{\pen}{\text{pen}}
\newcommand{\rg}{\text{rk}}
\newcommand{\Tr}{\text{Tr}}
\newcommand{\tra}[1]{\, {}^t\! #1}
\newcommand{\lspan}{\text{Span}}

\newtheorem{prpstn}{Proposition}[section]
\newtheorem{thrm}{Theorem}[section]
\newtheorem{crllr}{Corollary}[section]
\newtheorem{lmm}{Lemma}[section]

\title{Model selection and estimation\\of a component in additive regression}
\author{Xavier Gendre\vspace{0.2cm}\\Institut de Math\'ematiques de Toulouse\\Universit\'e de Toulouse et CNRS (UMR 5219)\vspace{0.2cm}\\{\small {\tt Xavier.Gendre@math.univ-toulouse.fr}}}
\date{}

\begin{document}

\maketitle

\begin{abstract}Let $Y\in\R^n$ be a random vector with mean $s$ and covariance matrix $\sigma^2P_n\tra{P_n}$ where $P_n$ is some known $n\times n$-matrix. We construct a statistical procedure to estimate $s$ as well as under moment condition on $Y$ or Gaussian hypothesis. Both cases are developed for known or unknown $\sigma^2$. Our approach is free from any prior assumption on $s$ and is based on non-asymptotic model selection methods. Given some linear spaces collection $\{S_m,\ m\in\M\}$, we consider, for any $m\in\M$, the least-squares estimator $\hat{s}_m$ of $s$ in $S_m$. Considering a penalty function that is not linear in the dimensions of the $S_m$'s, we select some $\hat{m}\in\M$ in order to get an estimator $\hat{s}_{\hat{m}}$ with a quadratic risk as close as possible to the minimal one among the risks of the $\hat{s}_m$'s. Non-asymptotic oracle-type inequalities and minimax convergence rates are proved for $\hat{s}_{\hat{m}}$. A special attention is given to the estimation of a non-parametric component in additive models. Finally, we carry out a simulation study in order to illustrate the performances of our estimators in practice.\end{abstract}

\section{Introduction}

\subsection{Additive models}

The general form of a \textit{regression model} can be expressed as
\begin{equation}\label{a_gen_f}
Z=f(X)+\sigma\varepsilon
\end{equation}
where $X=(X^{(1)},\dots,X^{(k)})'$ is the $k$-dimensional vector of \textit{explanatory variables} that belongs to some product space $\mathcal{X}=\mathcal{X}_1\times\dots\times\mathcal{X}_k\subset\R^k$, the unknown function $f:\mathcal{X}\rightarrow\R$ is called \textit{regression function}, the positive real number $\sigma$ is a standard deviation factor and the real random noise $\varepsilon$ is such that $\E[\varepsilon\vert X]=0$ and $\E[\varepsilon^2\vert X]<\infty$ almost surely.

In such a model, we are interested in the behavior of $Z$ in accordance with the fluctuations of $X$. In other words, we want to explain the random variable $Z$ through the function $f(x)=\E[Z\vert X=x]$. For this purpose, many approaches have been proposed and, among them, a widely used is the \textit{linear regression}
\begin{equation}\label{a_lin_mod}
Z=\mu+\sum_{i=1}^k\beta_iX^{(i)}+\sigma\varepsilon
\end{equation}
where $\mu$ and the $\beta_i$'s are unknown constants. This model benefits from easy interpretation in practice and, from a statistical point of view, allows componentwise analysis. However, a drawback of linear regression is its lack of flexibility for modeling more complex dependencies between $Z$ and the $X^{(i)}$'s. In order to bypass this problem while keeping the advantages of models like (\ref{a_lin_mod}), we can generalize them by considering \textit{additive regression models} of the form
\begin{equation}\label{a_general_frame}
Z=\mu+\sum_{i=1}^kf_i(X^{(i)})+\sigma\varepsilon
\end{equation}
where the unknown functions $f_i:\mathcal{X}_i\rightarrow\R$ will be referred to as the \textit{components} of the regression function $f$. The object of this paper is to construct a data-driven procedure for estimating one of these components on a fixed design (\textit{i.e.} conditionally to some realizations of the random variable $X$). Our approach is based on nonasymptotic model selection and is free from any prior assumption on $f$ and its components. In particular, we do not make any regularity hypothesis on the function to estimate except to deduce uniform convergence rates for our estimators.

Models (\ref{a_general_frame}) are not new and were first considered in the context of input-output analysis by Leontief \cite{Leo47} and in analysis of variance by Scheff\'e \cite{Sch59}. This kind of model structure is widely used in theoretical economics and in econometric data analysis and leads to many well known economic results. For more details about interpretability of additive models in economics, the interested reader could find many references at the end of Chapter 8 of \cite{HarMulSpeWer04}.

As we mention above, regression models are useful for interpreting the effects of $X$ on changes of $Z$. To this end, the statisticians have to estimate the regression function $f$. Assuming that we observe a sample $\{(X_1,Z_1),\dots,(X_n,Z_n)\}$ obtained from model (\ref{a_gen_f}), it is well known (see \cite{Sto85}) that the optimal $\Ltwo$ convergence rate for estimating $f$ is of order $n^{-\alpha/(2\alpha+k)}$ where $\alpha>0$ is an index of smoothness of $f$. Note that, for large value of $k$, this rate becomes slow and the performances of any estimation procedure suffer from what is called the \textit{curse of the dimension} in literature. In this connection, Stone \cite{Sto85} has proved the notable fact that, for additive models (\ref{a_general_frame}), the optimal $\Ltwo$ convergence rate for estimating each component $f_i$ of $f$ is the one-dimensional rate $n^{-\alpha/(2\alpha+1)}$. In other terms, estimation of the component $f_i$ in (\ref{a_general_frame}) can be done with the same optimal rate than the one achievable with the model $Z'=f_i(X^{(i)})+\sigma\varepsilon$.

Components estimation in additive models has received a large interest since the eighties and this theory benefited a lot from the the works of Buja \textit{et al.} \cite{BujHasTib89}, Hastie and Tibshirani \cite{HasTib90}. Very popular methods for estimating components in (\ref{a_general_frame}) are based on \textit{backfitting} procedures (see \cite{BreFri85} for more details). These techniques are iterative and may depend on the starting values. The performances of these methods deeply depends on the choice of some convergence criterion and the nature of the obtained results is usually asymptotic (see, for example, the works of Opsomer and Ruppert \cite{OpsRup97} and Mammen, Linton and Nielsen \cite{MamLinNie99}). More recent non-iterative methods have been proposed for estimating marginal effects of the $X^{(i)}$ on the variable $Z$ (\textit{i.e.} how $Z$ fluctuates on average if one explanatory variable is varying while others stay fixed). These procedures, known as \textit{marginal integration estimation}, were introduced by Tj{\o}stheim and Auestad \cite{TjoAue94} and Linton and Nielsen \cite{LinNie95}. In order to estimate the marginal effect of $X^{(i)}$, these methods take place in two times. First, they estimate the regression function $f$ by a particular estimator $f^*$, called \textit{pre-smoother}, and then they average $f^*$ according to all the variables except $X^{(i)}$. The way for constructing $f^*$ is fundamental and, in practice, one uses a special kernel estimator (see \cite{RupWan94} and \cite{SevSpe99} for a discussion on this subject). To this end, one needs to estimate two unknown bandwidths that are necessary for getting $f^*$. Dealing with a finite sample, the impact of how we estimate these bandwidths is not clear and, as for backfitting, the theoretical results obtained by these methods are mainly asymptotic.

In contrast with these methods, we are interested here in nonasymptotic procedures to estimate components in additive models. The following subsection is devoted to introduce some notations and the framework that we handle but also a short review of existing results in nonasymptotic estimation in additive models.

\subsection{Statistical framework}

We are interested in estimating one of the components in the model (\ref{a_general_frame}) with, for any $i$, $\mathcal{X}_i=[0,1]$. To focus on it, we denote by $s:[0,1]\rightarrow\R$ the component that we plan to estimate and by $t^1,\dots,t^K:[0,1]\rightarrow\R$ the $K\geqslant1$ other ones. Thus, considering the design points $(x_1,y^1_1,\dots,y^K_1)',\dots,(x_1,y^1_1,\dots,y^K_1)'\in[0,1]^{K+1}$, we observe
\begin{equation}\label{a_frame}
Z_i=s(x_i)+\mu+\sum_{j=1}^Kt^j(y^j_i)+\sigma\varepsilon_i,\ i=1,\dots,n\ ,
\end{equation}
where the components $s,t^1,\dots,t^K$ are unknown functions, $\mu$ in an unknown real number, $\sigma$ is a positive factor and $\varepsilon=(\varepsilon_1,\dots,\varepsilon_n)'$ is an unobservable centered random vector with i.i.d. components of unit variance.

Let $\nu$ be a probability measure on $[0,1]$, we introduce the space of centered and square-integrable functions
$$\Lp^2_0([0,1],\nu)=\left\{f\in\Lp^2([0,1],\nu)\ :\ \int_0^1f(t)\nu(dt)=0\right\}\ .$$
Let $\nu_1,\dots,\nu_K$ be $K$ probability measures on $[0,1]$, to avoid identification problems in the sequel, we assume
\begin{equation}
s\in\Lp^2_0\left([0,1],\nu\right)\text{ and }t^j\in\Lp^2_0\left([0,1],\nu_j\right),\ j=1,\dots,K\ .
\end{equation}
This hypothesis is not restrictive since we are interested in how $Z=(Z_1,\dots,Z_n)'$ fluctuates with respect to the $x_i$'s. A shift on the components does not affect these fluctuations and the estimation proceeds up to the additive constant $\mu$.

The results described in this paper are obtained under two different assumptions on the noise terms $\varepsilon_i$, namely

\vspace{\baselineskip}
\textbf{($\text{H}_{\text{Gau}}$)} the random vector $\varepsilon$ is a standard Gaussian vector in $\R^n$,

\vspace{\baselineskip}
\noindent and

\vspace{\baselineskip}
\textbf{($\text{H}_{\text{Mom}}$)} the variables $\varepsilon_i$ satisfy the moment condition
\begin{equation}\label{m_cond}
\exists p>2\text{ such that }\forall i,\ \tau_p=\E\left[\vert\varepsilon_i\vert^p\right]<\infty\ .
\end{equation}
Obviously, \textbf{($\text{H}_{\text{Mom}}$)} is weaker than \textbf{($\text{H}_{\text{Gau}}$)}. We consider these two cases in order to illustrate how better are the results in the Gaussian case with regard to the moment condition case. From the point of view of model selection, we show in the corollaries of Section \ref{a_section_known} that we are allowed to work with more general model collections under \textbf{($\text{H}_{\text{Gau}}$)} than under \textbf{($\text{H}_{\text{Mom}}$)} in order to get similar results. Thus, the main contribution of the Gaussian assumption is to give more flexibility to the procedure described in the sequel.

So, our aim is to estimate the component $s$ on the basis of the observations (\ref{a_frame}). For the sake of simplicity of this introduction, we assume that the quantity $\sigma^2>0$ is known (see Section \ref{a_unknown_var} for unknown variance) and we introduce the vectors $s=(s_1,\dots,s_n)'$ and $t=(t_1,\dots,t_n)'$ defined by, for any $i\in\{1,\dots,n\}$,
\begin{equation}\label{a_def_s_t}
\begin{array}{ccc}
s_i=s(x_i) & \text{and} & \displaystyle{t_i=\mu+\sum_{j=1}^Kt^j(y^j_i)\ .}
\end{array}
\end{equation}
Moreover, we assume that we know two linear subspaces $E,F\subset\R^n$ such that $s\in E$, $t\in F$ and $E\oplus F=\R^n$. Of course, such spaces are not available to the statisticians in practice and, when we handle additive models in Section \ref{sect_estim_comp}, we will not suppose that they are known. Let $P_n$ be the projection onto $E$ along $F$, we derive from (\ref{a_frame}) the following regression framework
\begin{equation}\label{a_my_frame}
Y=P_nZ=s+\sigma P_n\varepsilon
\end{equation}
where $Y=(Y_1,\dots,Y_n)'$ belongs to $E=\im(P_n)\subset\R^n$.

The framework (\ref{a_my_frame}) is similar to the classical \textit{signal-plus-noise} regression framework but the data are not independent and their variances are not equal. Because of this uncommonness of the variances of the observations, we qualify (\ref{a_my_frame}) as an \textit{heteroscedastic} framework. The object of this paper is to estimate the component $s$ and we handle (\ref{a_my_frame}) to this end. The particular case of $P_n$ equal to the unit matrix has been widely treated in the literature (see, for example, \cite{BirMas01} for \textbf{($\text{H}_{\text{Gau}}$)} and \cite{Bar00} for \textbf{($\text{H}_{\text{Mom}}$)}). The case of an unknown but diagonal matrix $P_n$ has been studied in several papers for the Gaussian case (see, for example, \cite{ComRoz02} and \cite{Gen08}). By using cross-validation and resampling penalties, Arlot and Massart \cite{ArlMas09} and Arlot \cite{Arl10} have also considered the framework \eqref{a_my_frame} with unknown diagonal matrix $P_n$. Laurent, Loubes and Marteau \cite{LauLouMar11} deal with a known diagonal matrix $P_n$ for studying testing procedure in an inverse problem framework. The general case of a known non-diagonal matrix $P_n$ naturally appears in applied fields as, for example, genomic studies (see Chapters 4 and 5 of \cite{RobRodSch05}).

The results that we introduce in the sequel consider the framework (\ref{a_my_frame}) from a general outlook and we do not make any prior hypothesis on $P_n$. In particular, we do not suppose that $P_n$ is invertible. We only assume that it is a projector when we handle the problem of component estimation in an additive framework in Section \ref{sect_estim_comp}. Without loss of generality, we always admit that $s\in\im(P_n)$. Indeed, if $s$ does not belong to $\im(P_n)$, it suffices to consider the orthogonal projection $\pi_{P_n}$ onto $\im(P_n)^{\perp}$ and to notice that $\pi_{P_n}Y=\pi_{P_n}s$ is not random. Thus, replacing $Y$ by $Y-\pi_{P_n}Y$ leads to (\ref{a_my_frame}) with a mean lying in $\im(P_n)$.  For general matrix $P_n$, other approaches could be used. However, for the sake of legibility, we consider $s\in\im(P_n)$ because, for the estimation of a component in an additive framework, by construction, we always have $Y=P_nZ\in\im(P_n)$ as it will be specified in Section \ref{sect_estim_comp}.

We now describe our estimation procedure in details. For any $z\in\R^n$, we define the \textit{least-squares contrast} by
$$\gamma_n(z)=\Vert Y-z\Vert_n^2=\frac{1}{n}\sum_{i=0}^n(Y_i-z_i)^2\ .$$
Let us consider a collection of linear subspaces of $\im(P_n)$ denoted by $\mathcal{F}=\{S_m,\ m\in\M\}$ where $\M$ is a finite or countable index set. Hereafter, the $S_m$'s will be called the \textit{models}. Denoting by $\pi_m$ the orthogonal projection onto $S_m$, the minimum of $\gamma_n$ over $S_m$ is achieved at a single point $\hat{s}_m=\pi_mY$ called the \textit{least-squares estimator} of $s$ in $S_m$. Note that the expectation of $\hat{s}_m$ is equal to the orthogonal projection $s_m=\pi_ms$ of $s$ onto $S_m$. We have the following identity for the quadratic risks of the $\hat{s}_m$'s,

\begin{prpstn}\label{prop_risk}
Let $m\in\M$, the least-squares estimator $\hat{s}_m=\pi_mY$ of $s$ in $S_m$ satisfies
\begin{equation}\label{a_bias_var}
\E\left[\Vert s-\hat{s}_m\Vert_n^2\right]=\Vert s-s_m\Vert_n^2+\frac{\Tr(\tra{P_n}\pi_mP_n)}{n}\sigma^2
\end{equation}
where $\Tr(\cdot)$ is the trace operator.
\end{prpstn}

\begin{proof}
By orthogonality, we have
\begin{equation}\label{prop_risk_decomp}
\Vert s-\hat{s}_m\Vert_n^2=\Vert s-s_m\Vert_n^2+\sigma^2\Vert\pi_mP_n\varepsilon\Vert_n^2\ .
\end{equation}
Because the components of $\varepsilon$ are independent and centered with unit variance, we easily compute
$$\E\left[\Vert\pi_mP_n\varepsilon\Vert_n^2\right]=\frac{\Tr(\tra{P_n}\pi_mP_n)}{n}\ .$$
We conclude by taking the expectation on both side of (\ref{prop_risk_decomp}).
\end{proof}

A ``good'' estimator is such that its quadratic risk is small. The decomposition given by (\ref{a_bias_var}) shows that this risk is a sum of two non-negative terms that can be interpreted as follows. The first one, called \textit{bias term}, corresponds to the capacity of the model $S_m$ to approximate the true value of $s$. The second, called \textit{variance term}, is proportional to $\Tr(\tra{P_n}\pi_mP_n)$ and measures, in a certain sense, the complexity of $S_m$. If $S_m=\R u$, for some $u\in\R^n$, then the variance term is small but the bias term is as large as $s$ is far from the too simple model $S_m$. Conversely, if $S_m$ is a ``huge'' model, whole $\R^n$ for instance, the bias is null but the price is a great variance term. Thus, (\ref{a_bias_var}) illustrates why choosing a ``good'' model amounts to finding a trade-off between bias and variance terms.

Clearly, the choice of a model that minimizes the risk (\ref{a_bias_var}) depends on the unknown vector $s$ and makes good models unavailable to the statisticians. So, we need a data-driven procedure to select an index $\hat{m}\in\M$ such that $\E[\Vert s-\hat{s}_{\hat{m}}\Vert_n^2]$ is close to the smaller $\Ltwo$-risk among the collection of estimators $\{\hat{s}_m,\ m\in\M\}$, namely
$$\mathcal{R}(s,\mathcal{F})=\inf_{m\in\M}\E\left[\Vert s-\hat{s}_m\Vert_n^2\right]\ .$$

To choose such a $\hat{m}$, a classical way in model selection consists in minimizing an empirical penalized criterion stochastically close to the risk. Given a \textit{penalty} function $\pen:\M\rightarrow\R_+$, we define $\hat{m}$ as any minimizer over $\M$ of the penalized least-squares criterion
\begin{equation}\label{pen_crit}
\hat{m}\in\argmin_{m\in\M}\left\{\gamma_n(\hat{s}_m)+\pen(m)\right\}\ .
\end{equation}
This way, we select a model $S_{\hat{m}}$ and we have at our disposal the \textit{penalized least-squares estimator} $\tilde{s}=\hat{s}_{\hat{m}}$. Note that, by definition, the estimator $\tilde{s}$ satisfies
\begin{equation}\label{prop_def_plse}
\forall m\in\M,\ \gamma_n(\tilde{s})+\pen(\hat{m})\leqslant\gamma_n(\hat{s}_m)+\pen(m)\ .
\end{equation}
To study the performances of $\tilde{s}$, we have in mind to upperbound its quadratic risk. To this end, we establish inequalities of the form
\begin{equation}\label{a_pseudo2}
\E\left[\Vert s-\tilde{s}\Vert_n^2\right]\leqslant C\inf_{m\in\M}\left\{\Vert s-s_m\Vert_n^2+\pen(m)\right\}+\frac{R}{n}
\end{equation}
where $C$ and $R$ are numerical terms that do not depend on $n$. Note that if the penalty is proportional to $\Tr(\tra{P_n}\pi_mP_n)\sigma^2/n$, then the quantity involved in the infimum is of order of the $\Ltwo$-risk of $\hat{s}_m$. Consequently, under suitable assumptions, such inequalities allow us to deduce upperbounds of order of the minimal risk among the collection of estimators $\{\hat{s}_m,\ m\in\M\}$. This result is known as an \textit{oracle inequality}
\begin{equation}\label{a_oracle}
\E\left[\Vert s-\tilde{s}\Vert_n^2\right]\leqslant C\mathcal{R}(s,\mathcal{F})=C\inf_{m\in\mathcal{M}}\E\left[\Vert s-\hat{s}_m\Vert_n^2\right]\ .
\end{equation}

This kind of procedure is not new and the first results in estimation by penalized criterion are due to Akaike \cite{Aka69} and Mallows \cite{Mal73} in the early seventies. Since these works, model selection has known an important development and it would be beyond the scope of this paper to make an exhaustive historical review of the domain. We refer to the first chapters of \cite{MacTsa98} for a more general introduction.

Nonasymptotic model selection approach for estimating components in an additive model was studied in few papers only. Considering penalties that are linear in the dimension of the models, Baraud, Comte and Viennet \cite{BarComVie01} have obtained general results for geometrically $\beta$-mixing regression models. Applying it to the particular case of additive models, they estimate the whole regression function. They obtain nonasymptotic upperbounds similar to (\ref{a_pseudo2}) on condition $\varepsilon$ admits a moment of order larger than 6. For additive regression on a random design and alike penalties, Baraud \cite{Bar02} proved oracle inequalities for estimators of the whole regression function constructed with polynomial collections of models and a noise that admits a moment of order 4. Recently, Brunel and Comte \cite{BruCom06} have obtained results with the same flavor for the estimation of the regression function in a censored additive model and a noise admitting a moment of order larger than 8. Pursuant to this work, Brunel and Comte \cite{BruCom08} have also proposed a nonasymptotic iterative method to achieve the same goal. Combining ideas from sparse linear modeling and additive regression, Ravikumar \textit{et al.} \cite{RavLiuLafWas09} have recently developed a data-driven procedure, called SpAM, for estimating a sparse high-dimensional regression function. Some of their empirical results have been proved by Meier, van de Geer and B\"uhlmann \cite{MeivdGBuh09} in the case of a sub-Gaussian noise and some sparsity-smoothness penalty.

The methods that we use are similar to the ones of Baraud, Comte and Viennet and are inspired from \cite{Bar00}. The main contribution of this paper is the generalization of the results of \cite{Bar00} and \cite{BarComVie01} to the framework \eqref{a_my_frame} with a known matrix $P_n$ under Gaussian hypothesis or only moment condition on the noise terms. Taking into account the correlations between the observations in the procedure leads us to deal with penalties that are not linear in the dimension of the models. Such a consideration naturally arises in heteroscedastic framework. Indeed, as mentioned in \cite{Arl10}, at least from an asymptotic point of view, considering penalties linear in the dimension of the models in an heteroscedastic framework does not lead to oracle inequalities for $\tilde{s}$. For our penalized procedure and under mild assumptions on $\mathcal{F}$, we prove oracle inequalities under Gaussian hypothesis on the noise or only under some moment condition.

Moreover, we introduce a nonasymptotic procedure to estimate one component in an additive framework. Indeed, the works cited above are all connected to the estimation of the whole regression function by estimating simultaneously all of its components. Since these components are each treated in the same way, their procedures can not focus on the properties of one of them. In the procedure that we propose, we can be sharper, from the point of view of the bias term, by using more models to estimate a particular component. This allows us to  deduce uniform convergence rates over H\"olderian balls and adaptivity of our estimators. Up to the best of our knowledge, our results in nonasymptotic estimation of a nonparametric component in an additive regression model are new.

\vspace{\baselineskip}
The paper is organized as follows. In Section \ref{a_section_known}, we study the properties of the estimation procedure under the hypotheses {\bf ($\text{H}_{\text{Gau}}$)} and {\bf ($\text{H}_{\text{Mom}}$)} with a known variance factor $\sigma^2$. As a consequence, we deduce oracle inequalities and we discuss about the size of the collection $\mathcal{F}$. The case of unknown $\sigma^2$ is presented in Section \ref{a_unknown_var} and the results of the previous section are extended to this situation. In Section \ref{sect_estim_comp}, we apply these results to the particular case of the additive models and, in the next section, we give uniform convergence rates for our estimators over H\"olderian balls. Finally, in Section \ref{a_section_simu1}, we illustrate the performances of our estimators in practice by a simulation study. The last sections are devoted to the proofs and to some technical lemmas.

\noindent{\bf Notations:} in the sequel, for any $x=(x_1,\dots,x_n)',\ y=(y_1,\dots,y_n)'\in\R^n$, we define
$$\Vert x\Vert_n^2=\frac{1}{n}\sum_{i=1}^nx_i^2\hspace{0.25cm}\text{and}\hspace{0.25cm}\langle x,y\rangle_n=\frac{1}{n}\sum_{i=1}^nx_iy_i\ .$$
We denote by $\rho$ the \textit{spectral norm} on the set $\Mn$ of the $n\times n$ real matrices as the norm induced by $\Vert\cdot\Vert_n$,
$$\forall A\in\Mn,\ \rho(A)=\sup_{x\in\R^n\setminus\{0\}}\frac{\Vert Ax\Vert_n}{\Vert x\Vert_n}\ .$$
For more details about the properties of $\rho$, see Chapter 5 of \cite{HorJoh90}.

\section{Main results}\label{a_section_known}

Throughout this section, we deal with the statistical framework given by (\ref{a_my_frame}) with $s\in\im(P_n)$ and we assume that the variance factor $\sigma^2$ is known. Moreover, in the sequel of this paper, for any $d\in\N$, we define $N_d$ as the number of models of dimension $d$ in $\mathcal{F}$,
$$N_d=\card\left\{m\in\M\ :\ \dim(S_m)=d\right\}\ .$$
We first introduce general model selection theorems under hypotheses {\bf ($\text{H}_{\text{Gau}}$)} and {\bf ($\text{H}_{\text{Mom}}$)}.

\begin{thrm}\label{G_main_theo}
Assume that {\bf ($\text{H}_{\text{Gau}}$)} holds and consider a collection of nonnegative numbers $\{L_m,m\in\M\}$. Let $\theta>0$, if the penalty function is such that
\begin{equation}\label{G_pen_theo_gen}
\pen(m)\geqslant(1+\theta+L_m)\frac{\Tr(\tra{P_n}\pi_mP_n)}{n}\sigma^2\text{ for all }m\in\M\ ,
\end{equation}
then the penalized least-squares estimator $\tilde{s}$ given by (\ref{pen_crit}) satisfies
\begin{equation}\label{G_pseudo_oracle}
\E\left[\Vert s-\tilde{s}\Vert_n^2\right] \leqslant C\inf_{m\in\M}\left\{\Vert s-s_m\Vert_n^2+\pen(m)-\frac{\Tr(\tra{P_n}\pi_mP_n)}{n}\sigma^2\right\}+\frac{\rho^2(P_n)\sigma^2}{n}R_n(\theta)
\end{equation}
where we have set
$$R_n(\theta)=C'\sum_{m\in\M}\exp\left(-\frac{C''L_m\Tr(\tra{P_n}\pi_mP_n)}{\rho^2(P_n)}\right)$$
and $C>1$ and $C',C''>0$ are constants that only depend on $\theta$.
\end{thrm}

If the errors are not supposed to be Gaussian but only to satisfy the moment condition {\bf ($\text{H}_{\text{Mom}}$)}, the following upperbound on the $q$-th moment of $\Vert s-\tilde{s}\Vert_n^2$ holds.

\begin{thrm}\label{NG_th1}
Assume that {\bf ($\text{H}_{\text{Mom}}$)} holds and take $q>0$ such that $2(q+1)<p$. Consider $\theta>0$ and some collection $\{L_m,\ m\in\M\}$ of positive weights. If the penalty function is such that
\begin{equation}\label{NG_th1_pen}
\pen(m)\geqslant(1+\theta+L_m)\frac{\Tr(\tra{P_n}\pi_mP_n)}{n}\sigma^2\text{ for all }m\in\M\ ,
\end{equation}
then the penalized least-squares estimator $\tilde{s}$ given by (\ref{pen_crit}) satisfies
\begin{equation}\label{NG_th1_res}
\E\left[\Vert s-\tilde{s}\Vert_n^{2q}\right]^{1/q} \leqslant C\inf_{m\in\M}\left\{\Vert s-s_m\Vert_n^2+\pen(m)\right\}+\frac{\rho^2(P_n)\sigma^2}{n}R_n(p,q,\theta)^{1/q}
\end{equation}
where we have set $R_n(p,q,\theta)$ equal to
$$C'\tau_p\left[N_0 + \sum_{m\in\M:S_m\neq\{0\}}\left(1+\frac{\Tr(\tra{P_n}\pi_mP_n)}{\rho^2(\pi_mP_n)}\right)\left(\frac{L_m\Tr(\tra{P_n}\pi_mP_n)}{\rho^2(P_n)}\right)^{q-p/2}\right]$$
and $C=C(q,\theta)$, $C'=C'(p,q,\theta)$ are positive constants.
\end{thrm}

The proofs of these theorems give explicit values for the constants $C$ that appear in the upperbounds. In both cases, these constants go to infinity as $\theta$ tends to $0$ or increases toward infinity. In practice, it does neither seem reasonable to choose $\theta$ close to 0 nor very large. Thus this explosive behavior is not restrictive but we still have to choose a ``good'' $\theta$. The values for $\theta$ suggested by the proofs are around the unity but we make no claim of optimality. Indeed, this is a hard problem to determine an optimal choice for $\theta$ from theoretical computations since it could depend on all the parameters and on the choice of the collection of models. In order to calibrate it in practice, several solutions are conceivable. We can use a simulation study, deal with cross-validation or try to adapt the slope heuristics described in \cite{BirMas07} to our procedure.

For penalties of order of $\Tr(\tra{P_n}\pi_mP_n)\sigma^2/n$, Inequalities (\ref{G_pseudo_oracle}) and (\ref{NG_th1_res}) are not far from being oracle. Let us denote by $R_n$ the remainder term $R_n(\theta)$ or $R_n(p,q,\theta)$ according to whether {\bf ($\text{H}_{\text{Gau}}$)} or {\bf ($\text{H}_{\text{Mom}}$)} holds. To deduce oracle inequalities from that, we need some additional hypotheses as the following ones:

\vspace{\baselineskip}
\textbf{($\text{A}_1$)} there exists some universal constant $\zeta>0$ such that
$$\pen(m)\leqslant \zeta\frac{\Tr(\tra{P_n}\pi_mP_n)}{n}\sigma^2\text{, for all }m\in\M\ ,$$

\textbf{($\text{A}_2$)} there exists some constant $R>0$ such that
$$\sup_{n\geqslant 1}R_n\leqslant R\ ,$$

\textbf{($\text{A}_3$)} there exists some constant $\rho>1$ such that
$$\sup_{n\geqslant 1}\rho^2(P_n)\leqslant \rho^2\ .$$

Thus, under the hypotheses of Theorem \ref{G_main_theo} and these three assumptions, we deduce from (\ref{G_pseudo_oracle}) that
$$\E\left[\Vert s-\tilde{s}\Vert_n^2\right]\leqslant C\inf_{m\in\M}\left\{\Vert s-s_m\Vert_n^2+\frac{\Tr(\tra{P_n}\pi_mP_n)}{n}\sigma^2\right\}+\frac{R\rho^2\sigma^2}{n}$$
where $C$ is a constant that does not depend on $s$, $\sigma^2$ and $n$. By Proposition \ref{prop_risk}, this inequality corresponds to (\ref{a_oracle}) up to some additive term. To derive similar inequality from (\ref{NG_th1_res}), we need on top of that to assume that $p>4$ in order to be able to take $q=1$.

Assumption {\bf ($\text{A}_3$)} is subtle and strongly depends on the nature of $P_n$. The case of oblique projector that we use to estimate a component in an additive framework will be discussed in Section \ref{sect_estim_comp}. Let us replace it, for the moment, by the following one

\vspace{\baselineskip}
\textbf{($\text{A}_3'$)} there exists $c\in(0,1)$ that does not depend on $n$ such that
$$c\rho^2(P_n)\dim(S_m)\leqslant\Tr(\tra{P_n}\pi_mP_n)\ .$$

By the properties of the norm $\rho$, note that $\Tr(\tra{P_n}\pi_mP_n)$ always admits an upperbound with the same flavor
\begin{eqnarray*}
\Tr(\tra{P_n}\pi_mP_n) & = & \Tr(\pi_mP_n\tra{(\pi_mP_n)})\\
& \leqslant & \rho(\pi_mP_n\tra{(\pi_mP_n)})\rg(\pi_mP_n\tra{(\pi_mP_n)})\\
& \leqslant & \rho^2(\pi_mP_n)\rg(\pi_m)\\
& \leqslant & \rho^2(P_n)\dim(S_m)\ .
\end{eqnarray*}
In all our results, the quantity $\Tr(\tra{P_n}\pi_mP_n)$ stands for a dimensional term relative to $S_m$. Hypothesis {\bf ($\text{A}_3'$)} formalizes that by assuming that its order is the dimension of the model $S_m$ up to the norm of the covariance matrix $\tra{P_n}P_n$.

Let us now discuss about the assumptions {\bf ($\text{A}_1$)} and {\bf ($\text{A}_2$)}. They are connected and they raise the impact of the complexity of the collection $\mathcal{F}$ on the estimation procedure. Typically, condition {\bf ($\text{A}_2$)} will be fulfilled under {\bf ($\text{A}_1$)} when $\mathcal{F}$ is not too ``large'', that is, when the collection does not contain too many models with the same dimension. We illustrate this phenomenon by the two following corollaries.

\begin{crllr}\label{G_gen_coro1}
Assume that {\bf ($\text{H}_\text{Gau}$)} and {\bf ($\text{A}_3'$)} hold and consider some finite $A\geqslant0$ such that
\begin{equation}\label{a_expcondi}
\sup_{d\in\N:N_d>0}\frac{\log N_d}{d}\leqslant A\ .
\end{equation}
Let $L$, $\theta$ and $\omega$ be some positive numbers that satisfy
$$L\geqslant\frac{2(1+\theta)^3}{c\theta^2}(A+\omega)\ .$$
Then, the estimator $\tilde{s}$ obtained from (\ref{pen_crit}) with penalty function given by
$$\pen(m)=(1+\theta+L)\frac{\Tr(\tra{P_n}\pi_mP_n)}{n}\sigma^2$$
is such that
\begin{equation*}
\E\left[\Vert s-\tilde{s}\Vert_n^2\right]\leqslant C\inf_{m\in\M}\left\{\Vert s-s_m\Vert_n^2+(L\vee1)\frac{\Tr(\tra{P_n}\pi_mP_n)\vee (c\rho^2(P_n))}{n}\sigma^2\right\}
\end{equation*}
where $C>1$ only depends on $\theta$, $\omega$ and $c$.
\end{crllr}

\noindent For errors that only satisfy moment condition, we have the following similar result.

\begin{crllr}\label{NG_gen_coro1}
Assume that {\bf ($\text{H}_\text{Mom}$)} and {\bf ($\text{A}_3'$)} hold with $p>6$ and let $A>0$ and $\omega>0$ such that
\begin{equation}\label{a_polcondi}
N_0\leqslant1\hspace{0.5cm}\text{and}\hspace{0.5cm}\sup_{d>0:N_d>0}\frac{N_d}{(1+d)^{p/2-3-\omega}}\leqslant A\ .
\end{equation}
Consider some positive numbers $L$, $\theta$ and $\omega'$ that satisfy
$$L\geqslant\omega'A^{2/(p-2)}\ ,$$
then, the estimator $\tilde{s}$ obtained from (\ref{pen_crit}) with penalty function given by
$$\pen(m)=(1+\theta+L)\frac{\Tr(\tra{P_n}\pi_mP_n)}{n}\sigma^2$$
is such that
\begin{equation*}
\E\left[\Vert s-\tilde{s}\Vert_n^2\right]\leqslant C\tau_p\inf_{m\in\M}\left\{\Vert s-s_m\Vert_n^2+(L\vee1)\frac{\Tr(\tra{P_n}\pi_mP_n)\vee (c\rho^2(P_n))}{n}\sigma^2\right\}
\end{equation*}
where $C>1$ only depends on $\theta$, $p$, $\omega$, $\omega'$ and $c$.
\end{crllr}
Note that the assumption {\bf ($\text{A}_3'$)} guarantees that $\Tr(\tra{P_n}\pi_mP_n)$ is not smaller than $c\rho^2(P_n)\dim(S_m)$ and, at least for the models with positive dimension, this implies $\Tr(\tra{P_n}\pi_mP_n)\geqslant c\rho^2(P_n)$. Consequently, up to the factor $L$, the upperbounds of $\E\left[\Vert s-\tilde{s}\Vert_n^2\right]$ given by Corollaries \ref{G_gen_coro1} and \ref{NG_gen_coro1} are of order of the minimal risk $\mathcal{R}(s,\mathcal{F})$. To deduce oracle inequalities for $\tilde{s}$ from that, {\bf ($\text{A}_1$)} needs to be fulfilled. In other terms, we need to be able to consider some $L$ independently from the size $n$ of the data. It will be the case if the same is true for the bounds $A$.

Let us assume that the collection $\mathcal{F}$ is small in the sense that, for any $d\in\N$, the number of models $N_d$ is bounded by some constant term that neither depends on $n$ nor $d$. Typically, collections of nested models satisfy that. In this case, we are free to take $L$ equal to some universal constant. So, {\bf ($\text{A}_1$)} is true for $\zeta=1+\theta+L$ and oracle inequalities can be deduced for $\tilde{s}$. Conversely, a large collection $\mathcal{F}$ is such that there are many models with the same dimension. We consider that this situation happens, for example, when the order of $A$ is $\log n$. In such a case, we need to choose $L$ of order $\log n$ too and the upperbounds on the risk of $\tilde{s}$ become oracle type inequalities up to some logarithmic factor. However, we know that in some situations, this factor can not be avoided as in the complete variable selection problem with Gaussian errors (see Chapter 4 of \cite{Mas07}).

As a consequence, the same model selection procedure allows us to deduce oracle type inequalities under \textbf{($\text{H}_\text{Gau})$} and \textbf{($\text{H}_\text{Mom}$)}. Nevertheless, the assumption on $N_d$ in Corollary \ref{NG_gen_coro1} is more restrictive than the one in Corollary \ref{G_gen_coro1}. Indeed, to obtain an oracle inequality in the Gaussian case, the quantity $N_d$ is limited by $e^{Ad}$ while the bound is only polynomial in $d$ under moment condition. Thus, the Gaussian assumption {\bf ($\text{H}_\text{Gau}$)} allows to obtain oracle inequalities for more general collections of models.

\section{Estimation when variance is unknown}\label{a_unknown_var}

In contrast with Section \ref{a_section_known}, the variance factor $\sigma^2$ is here assumed to be unknown in (\ref{a_my_frame}). Since the penalties given by Theorems \ref{G_main_theo} and \ref{NG_th1} depend on $\sigma^2$, the procedure introduced in the previous section does not remain available to the statisticians. Thus, we need to estimate $\sigma^2$ in order to replace it in the penalty functions. The results of this section give upperbounds for the $\Ltwo$-risk of the estimators $\tilde{s}$ constructed in such a way.

To estimate the variance factor, we use a residual least-squares estimator $\hat{\sigma}^2$ that we define as follows. Let $V$ be some linear subspace of $\text{Im}(P_n)$ such that
\begin{equation}\label{GUV_hyp_dim}
\Tr(\tra{P_n}\pi P_n)\leqslant\Tr(\tra{P_n}P_n)/2
\end{equation}
where $\pi$ is the orthogonal projection onto $V$. We define
\begin{equation}\label{GUV_def_estim}
\hat{\sigma}^2=\frac{n\Vert Y-\pi Y\Vert_n^2}{\Tr\left(\tra{P_n}(I_n-\pi)P_n\right)}\ .
\end{equation}
First, we assume that the errors are Gaussian. The following result holds.

\begin{thrm}\label{GUV_main_theo}
Assume that {\bf ($\text{H}_\text{Gau}$)} holds. For any $\theta>0$, we define the penalty function
\begin{equation}\label{GUV_pen_def}
\forall m\in\M,\ \pen(m)=(1+\theta)\frac{\Tr(\tra{P_n}\pi_mP_n)}{n}\hat{\sigma}^2\ .
\end{equation}
Then, for some positive constants $C$, $C'$ and $C''$ that only depend on $\theta$, the penalized least-squares estimator $\tilde{s}$ satisfies
\begin{equation}\label{GUV_main_theo_ineg}
\E\left[\Vert s-\tilde{s}\Vert_n^2\right] \leqslant C\left(\inf_{m\in\M}\E\left[\Vert s-\hat{s}_m\Vert_n^2\right]+\Vert s-\pi s\Vert_n^2\right)+\frac{\rho^2(P_n)\sigma^2}{n}\bar{R}_n(\theta)
\end{equation}
where we have set
\begin{equation*}
\bar{R}_n(\theta)=C'\left[\left(2+\frac{\Vert s\Vert_n^2}{\rho^2(P_n)\sigma^2}\right)\exp\left(-\frac{\theta^2\Tr(\tra{P_n}P_n)}{32\rho^2(P_n)}\right)+\sum_{m\in\M}\exp\left(-C''\frac{\Tr(\tra{P_n}\pi_mP_n)}{\rho^2(P_n)}\right)\right]\ .
\end{equation*}
\end{thrm}

\noindent If the errors are only assumed to satisfy a moment condition, we have the following theorem.

\begin{thrm}\label{NG_UV_th}
Assume that {\bf ($\text{H}_\text{Mom}$)} holds. Let $\theta>0$, we consider the penalty function defined by
\begin{equation}\label{NG_UV_th_pen}
\forall m\in\M,\ \pen(m)=(1+\theta)\frac{\Tr(\tra{P_n}\pi_mP_n)}{n}\hat{\sigma}^2\ .
\end{equation}
For any $0<q\leqslant1$ such that $2(q+1)<p$, the penalized least-squares estimator $\tilde{s}$ satisfies
$$\E[\Vert s-\tilde{s}\Vert_n^{2q}]^{1/q}\leqslant C\left(\inf_{m\in\M}\E[\Vert s-\hat{s}_m\Vert_n^2]+2\Vert s-\pi s\Vert_n^2\right)+\rho^2(P_n)\sigma^2\bar{R}_n(p,q,\theta)$$
where $C=C(q,\theta)$ and $C'=C'(p,q,\theta)$ are positive constants, $\bar{R}_n(p,q,\theta)$ is equal to
$$\frac{R_n(p,q,\theta)^{1/q}}{n}+C'\tau_p^{1/q}\kappa_n\left(\frac{\Vert s\Vert_n^2}{\rho^2(P_n)\sigma^2}+\tau_p\right)\left(\frac{\rho^{2\alpha_p}(P_n)}{\Tr(\tra{P_n}P_n)^{\beta_p}}\right)^{1/q-2/p}$$
with $R_n(p,q,\theta)$ defined as in Theorem \ref{NG_th1}, $(\kappa_n)_{n\in\N}=(\kappa_n(p,q,\theta))_{n\in\N}$ is a sequence of positive numbers that tends to $\kappa=\kappa(p,q,\theta)>0$ as $\Tr(\tra{P_n}P_n)/\rho^2(P_n)$ increases toward infinity and
$$\alpha_p=(p/2-1)\vee1\text{ and }\beta_p=(p/2-1)\wedge1\ .$$
\end{thrm}

Penalties given by (\ref{GUV_pen_def}) and (\ref{NG_UV_th_pen}) are random and allow to construct estimators $\tilde{s}$ when $\sigma^2$ is unknown. This approach leads to theoretical upperbounds for the risk of $\tilde{s}$. Note that we use some generic model $V$ to construct $\hat{\sigma}^2$. This space is quite arbitrary and is pretty much limited to be an half-space of $\im(P_n)$. The idea is that taking $V$ as some ``large'' space can lead to a good approximation of the true $s$ and, thus, $Y-\pi Y$ is not far from being centered and its normalized norm is of order $\sigma^2$. However, in practice, it is known that the estimator $\hat{\sigma}^2$ inclined to overestimate the true value of $\sigma^2$ as illustrated by Lemmas \ref{G_UV_lem_conc} and \ref{NG_UV_lem_conc}. Consequently, the penalty function tends to be larger and the procedure overpenalizes models with high dimension. To offset this phenomenon, a practical solution could be to choose some smaller $\theta$ when $\sigma^2$ is unknown than when it is known as we discuss in Section \ref{a_section_simu1}.

\section{Application to additive models}\label{sect_estim_comp}

In this section, we focus on the framework (\ref{a_frame}) given by an additive model. To describe the procedure to estimate the component $s$, we assume that the variance factor $\sigma^2$ is known but it can be easily generalized to the unknown factor case by considering the results of Section \ref{a_unknown_var}. We recall that $s\in\Lp_0^2([0,1],\nu)$, $t^j\in\Lp_0^2([0,1],\nu_j),\ j=1,\dots,K$, and we observe
\begin{equation}\label{a_synt_frame}
Z_i=s_i+t_i+\sigma\varepsilon_i,\ i=1,\dots,n\ ,
\end{equation}
where the random vector $\varepsilon=(\varepsilon_1,\dots,\varepsilon_n)'$ is such that {\bf ($\text{H}_\text{Gau}$)} or {\bf ($\text{H}_\text{Mom}$)} holds and the vectors $s=(s_1,\dots,s_n)'$ and $t=(t_1,\dots,t_n)'$ are defined in (\ref{a_def_s_t}).

Let $\mathcal{S}_n$ be a linear subspace of $\Lp_0^2([0,1],\nu)$ and, for all $j\in\{1,\dots,K\}$, $\mathcal{S}^j_n$ be a linear subspace of $\Lp_0^2([0,1],\nu_j)$. We assume that these spaces have finite dimensions $D_n=\dim(\mathcal{S}_n)$ and $D_n^{(j)}=\dim(\mathcal{S}^j_n)$ such that
$$D_n+D^{(1)}_n+\dots+D^{(K)}_n< n\ .$$
We consider an orthonormal basis $\{\phi_1,\dots,\phi_{D_n}\}$ (resp. $\{\psi^{(j)}_1,\dots,\psi^{(j)}_{D^{(j)}_n}\}$) of $\mathcal{S}_n$ (resp. $\mathcal{S}^j_n$) equipped with the usual scalar product of $\Lp^2([0,1],\nu)$ (resp. of $\Lp^2([0,1],\nu_j)$). The linear spans $E,F^1,\dots,F^K\subset\R^n$ are defined by
$$E=\lspan\left\{(\phi_i(x_1),\dots,\phi_i(x_n))',\ i=1,\dots,D_n\right\}$$
and
$$F^j=\lspan\left\{(\psi^{(j)}_i(y^j_1),\dots,\psi^{(j)}_i(y^j_n))',\ i=1,\dots,D^{(j)}_n\right\},\ j=1,\dots,K\ .$$
Let $\textbf{1}_n=(1,\dots,1)'\in\R^n$, we also define
$$F=\R\textbf{1}_n+F^1+\dots+F^K$$
where $\R\textbf{1}_n$ is added to the $F^j$'s in order to take into account the constant part $\mu$ of (\ref{a_frame}). Furthermore, note that the sum defining the space $F$ does not need to be direct.

We are free to choose the functions $\phi_i$'s and $\psi^j_i$'s. In the sequel, we assume that these functions are chosen in such a way that the mild assumption $E\cap F=\{0\}$ is fulfilled. Note that we do not assume that $s$ belongs to $E$ neither that $t$ belongs to $F$. Let $G$ be the space $(E+F)^{\perp}$, we obviously have $E\oplus F\oplus G=\R^n$ and we denote by $P_n$ the projection onto $E$ along $F+G$. Moreover, we define $\pi_E$ and $\pi_{F+G}$ as the orthogonal projections onto $E$ and $F+G$ respectively. Thus, we derive the following framework from (\ref{a_synt_frame}),
\begin{equation}\label{a_frame_partial}
Y=P_nZ=\bar{s}+\sigma P_n\varepsilon
\end{equation}
where we have set
\begin{eqnarray*}
\bar{s} & = & P_ns+P_nt\\
& = & s+(P_n-I_n)s+P_nt\\
& = & s+(P_n-I_n)(s-\pi_Es)+P_n(t-\pi_{F+G}t)=s+h\ .
\end{eqnarray*}
Let $\mathcal{F}=\{S_m,\ m\in\M\}$ be a finite collection of linear subspaces of $E$, we apply the procedure described in Section \ref{a_section_known} to $Y$ given by (\ref{a_frame_partial}), that is, we choose an index $\hat{m}\in\M$ as a minimizer of (\ref{pen_crit}) with a penalty function satisfying the hypotheses of Theorems \ref{G_main_theo} or \ref{NG_th1} according to whether {\bf ($\text{H}_\text{Gau}$)} or {\bf ($\text{H}_\text{Mom}$)} holds. This way, we estimate $s$ by $\tilde{s}$. From the triangular inequality, we derive that
$$\E[\Vert s-\tilde{s}\Vert_n^2]\leqslant2\E[\Vert\bar{s}-\tilde{s}\Vert_n^2]+2\Vert h\Vert_n^2\ .$$
As we discussed previously, under suitable assumptions on the complexity of the collection $\mathcal{F}$, we can assume that {\bf ($\text{A}_1$)} and {\bf ($\text{A}_2$)} are fulfilled. Let us suppose for the moment that {\bf ($\text{A}_3$)} is satisfied for some $\rho>1$. Note that, for any $m\in\M$, $\pi_m$ is an orthogonal projection onto the image set of the oblique projection $P_n$. Consequently, we have $\Tr(\tra{P_n}\pi_mP_n)\geqslant\rg(\pi_m)=\dim(S_m)$ and Assumption {\bf ($\text{A}_3$)} implies {\bf ($\text{A}_3'$)} with $c=1/\rho^2$. Since, for all $m\in\M$,
$$\Vert \bar{s}-\pi_m\bar{s}\Vert_n\leqslant\Vert s-\pi_ms\Vert_n+\Vert h-\pi_mh\Vert_n\leqslant\Vert s-\pi_ms\Vert_n+\Vert h\Vert_n\ ,$$
we deduce from Theorems \ref{G_main_theo} or \ref{NG_th1} that we can find, independently from $s$ and $n$, two positive numbers $C$ and $C'$ such that
\begin{equation}\label{a_ineg01}
\E[\Vert s-\tilde{s}\Vert_n^2]\leqslant C\inf_{m\in\M}\left\{\Vert s-\pi_ms\Vert_n^2+\frac{\Tr(\tra{P_n}\pi_mP_n)}{n}\sigma^2\right\}+C'\left(\Vert h\Vert_n^2+\frac{\rho^2\sigma^2}{n}R\right)\ .
\end{equation}
To derive an interesting upperbound on the $\Ltwo$-risk of $\tilde{s}$, we need to control the remainder term. Because $\rho(\cdot)$ is a norm on $\Mn$, we dominate the norm of $h$ by
\begin{eqnarray*}
\Vert h\Vert_n & \leqslant & \rho(I_n-P_n)\Vert s-\pi_Es\Vert_n+\rho(P_n)\Vert t-\pi_{F+G}t\Vert_n\\
& \leqslant & (1+\rho(P_n))(\Vert s-\pi_Es\Vert_n+\Vert t-\pi_{F+G}t\Vert_n)\\
& \leqslant & (1+\rho)(\Vert s-\pi_Es\Vert_n+\Vert t-\pi_{F+G}t\Vert_n)\ .
\end{eqnarray*}
Note that, for any $m\in\M$, $S_m\subset E$ and so, $\Vert s-\pi_Es\Vert_n\leqslant\Vert s-\pi_ms\Vert_n$. Thus, Inequality (\ref{a_ineg01}) leads to
\begin{equation}\label{a_ineg02}
\E[\Vert s-\tilde{s}\Vert_n^2]\leqslant C(1+\rho)^2\inf_{m\in\M}\left\{\Vert s-\pi_ms\Vert_n^2+\frac{\Tr(\tra{P_n}\pi_mP_n)}{n}\sigma^2\right\}+C'(1+\rho)^2\left(\Vert t-\pi_{F+G}t\Vert_n^2+\frac{\sigma^2}{n}R\right)\ .
\end{equation}
The space $F+G$ has to be seen as a large approximation space. So, under a reasonable assumption on the regularity of the component $t$, the quantity $\Vert t-\pi_{F+G}t\Vert_n^2$ could be regarded as being neglectable. It mainly remains to understand the order of the multiplicative factor $(1+\rho)^2$.

\vspace{\baselineskip}

Thus, we now discuss about the norm $\rho(P_n)$ and the assumption {\bf ($\text{A}_3$)}. This quantity depends on the design points $(x_i, y^1_i,\dots,y^K_i)\in[0,1]^{K+1}$ and on how we construct the spaces $E$ and $F$, \textit{i.e.} on the choice of the basis functions $\phi_i$ and $\psi^{(j)}_i$. Hereafter, the design points $(x_i, y^1_i,\dots,y^K_i)$ will be assumed to be known independent realizations of a random variable on $[0,1]^{K+1}$ with distribution $\nu\otimes\nu_1\otimes\dots\otimes\nu_K$. We also assume that these points are independent of the noise $\varepsilon$ and we proceed conditionally to them. To discuss about the probability for {\bf ($\text{A}_3$)} to occur, we introduce some notations. We denote by $D'_n$ the integer
$$D'_n=1+D^{(1)}_n+\dots+D^{(K)}_n$$
and we have $\dim(F)\leqslant D'_n$. Let $A$ be a $p\times p$ real matrix, we define
$$r_p(A)=\sup\left\{\sum_{i=1}^p\sum_{j=1}^p\vert a_ia_j\vert\times\vert A_{ij}\vert\ : \sum_{i=1}^pa_i^2\leqslant1\right\}\ .$$
Moreover, we define the matrices $V(\phi)$ and $B(\phi)$ by
$$
\begin{array}{ccc}
\displaystyle{V_{ij}(\phi)=\sqrt{\int_0^1\phi_i(x)^2\phi_j(x)^2\nu(dx)}} & \text{and} & \displaystyle{B_{ij}(\phi)=\sup_{x\in[0,1]}\vert\phi_i(x)\phi_j(x)\vert\ ,}
\end{array}
$$
for any $1\leqslant i,j\leqslant D_n$. Finally, we introduce the quantities
$$
\begin{array}{ccc}
L_{\phi}=\max\left\{r_{D_n}^2(V(\phi)),\ r_{D_n}(B(\phi))\right\} & \text{and} & \displaystyle{b_{\phi}=\max_{i=1,\dots,D_n}\sup_{x\in[0,1]}\vert\phi_i(x)\vert}
\end{array}
$$
and
$$L_n=\max\left\{L_{\phi},\ D_nD_n',\ b_{\phi}\sqrt{nD_nD'_n}\right\}\ .$$
\begin{prpstn}\label{prop_rho}
Consider the matrix $P_n$ defined in (\ref{a_frame_partial}). We assume that the design points are independent realizations of a random variable on $[0,1]^{K+1}$ with distribution $\nu\otimes\nu_1\otimes\dots\otimes\nu_K$ such that we have $E\cap F=\{0\}$ and $\dim(E)=D_n$ almost surely. If the basis $\{\phi_1,\dots,\phi_{D_n}\}$ is such that
\begin{equation}\label{condi_phi}
\forall 1\leqslant i\leqslant D_n,\ \int_0^1\phi_i(x)\nu(dx)=0
\end{equation}
then, there exists some universal constant $C>0$ such that, for any $\rho>1$,
$$\Pg\left(\rho(P_n)>\rho\right)\leqslant 4D_n(D_n+D'_n)\exp\left(-\frac{Cn}{L_n}(1-\rho^{-1})^2\right)\ .$$
\end{prpstn}
As a consequence of Proposition \ref{prop_rho}, we see that {\bf ($\text{A}_3$)} is fulfilled with a large probability since we choose basis functions $\phi_i$ in such a way to keep $L_n$ small in front of $n$. It will be so for localized bases (piecewise polynomials, orthonormal wavelets, ...) with $L_n$ of order of $n^{1-\omega}$, for some $\omega\in(0,1)$, once we consider $D_n$ and $D_n'$ of order of $n^{\frac{1}{3}-\frac{3\omega}{2}}$ (this is a direct consequence of Lemma $1$ in \cite{BirMas98}). This limitation, mainly due to the generality of the proposition, could seem restrictive from a practical point of view. However the statistician can explicitly compute $\rho(P_n)$ with the data. Thus, it is possible to adjust $D_n$ and $D_n'$ in order to keep $\rho(P_n)$ small in practice. Moreover, we will see in Section \ref{a_section_simu1} that, for our choices of $\phi_i$ and $\psi^j_i$, we can easily consider $D_n$ and $D_n'$ of order of $\sqrt{n}$ as we keep $\rho(P_n)$ small (concrete values are given in the simulation study).

\section{Convergence rates}

The previous sections have introduced various upperbounds on the $\Ltwo$-risk of the penalized least-squares estimators $\tilde{s}$. Each of them is connected to the minimal risk of the estimators among a collection $\{\hat{s}_m,m\in\M\}$. One of the main advantages of such inequalities is that it allows us to derive uniform convergence rates with respect to many well known classes of smoothness (see \cite{BirMas97}). In this section, we give such results over H\"olderian balls for the estimation of a component in an additive framework. To this end, for any $\alpha>0$ and $R>0$, we introduce the space $\mathcal{H}_{\alpha}(R)$ of the $\alpha$-H\"olderian functions with constant $R>0$ on $[0,1]$,
$$\mathcal{H}_{\alpha}(R)=\left\{f:[0,1]\rightarrow\R\ :\ \forall x,y\in[0,1],\ \vert f(x)-f(y)\vert\leqslant R\vert x-y\vert^{\alpha}\right\}\ .$$

In order to derive such convergence rates, we need a collection of models $\mathcal{F}$ with good approximation properties for the functions of $\mathcal{H}_{\alpha}(R)$. We denote by $P_n^{BM}$ any oblique projector defined as in the previous section and based on spaces $\mathcal{S}_n$ and $\mathcal{S}_n^j$ that are constructed as one of the examples given in Section 2 of \cite{BirMas00}. In particular, such a construction allows us to deal with approximation spaces $\mathcal{S}_n$ and $\mathcal{S}_n^j$ that can be considered as spaces of piecewise polynomials, spaces of orthogonal wavelet expansions or spaces of dyadic splines on $[0,1]$. We consider the dimensions $D_n=\dim(\mathcal{S}_n)$ and, for any $j\in\{1,\dots,K\}$, $D_n^{(j)}=\dim(\mathcal{S}_n^j)=D_n/K$. Finally, we take a collection of models $\mathcal{F}^{BM}$ that contains subspaces of $E=\im(P_n^{BM})$ as Baraud did in Section 2.2 of \cite{Bar02}.

\begin{prpstn}\label{a_prop_rate}
Consider the framework (\ref{a_frame}) and assume that {\bf ($\text{H}_\text{Gau}$)} or {\bf ($\text{H}_\text{Mom}$)} holds with $p>6$. We define $Y$ in (\ref{a_frame_partial}) with $P_n^{BM}$. Let $\eta>0$ and $\tilde{s}$ be the estimator selected by the procedure (\ref{pen_crit}) applied to the collection of models $\mathcal{F}^{BM}$ with the penalty
$$pen(m)=(1+\eta)\frac{\Tr(\tra{P_n^{BM}}\pi_mP_n^{BM})}{n}\sigma^2\ .$$
Suppose that {\bf ($\text{A}_3$)} is fulfilled, we define
$$\zeta_n=\frac{1}{2}\left(\frac{\log n}{\log D_n}-1\right)>0\ .$$
For any $\alpha>\zeta_n$ and $R>0$, the penalized least-squares estimator $\tilde{s}$ satisfies
\begin{equation}\label{a_opt_rate}
\sup_{(s,t^1,\dots,t^K)\in\mathcal{H}_{\alpha}(R)^{K+1}}\E_{\varepsilon,d}\left[\Vert s-\tilde{s}\Vert_n^2\right]\leqslant C_{\alpha}n^{-2\alpha/(2\alpha+1)}
\end{equation}
where $\E_{\varepsilon,d}$ is the expectation on $\varepsilon$ and on the random design points and $C_{\alpha}>1$ only depends on $\alpha$, $\rho$, $\sigma^2$, $K$, $L$, $\theta$ and $p$ (under {\bf ($\text{H}_\text{Mom}$)} only).
\end{prpstn}

Note that the supremum is taken over H\"olderian balls for all the components of the regression function, \textit{i.e.} the regression function is itself supposed to belong to an H\"olderian space. As we mention in the introduction, Stone \cite{Sto85} has proved that the rate of convergence given by (\ref{a_opt_rate}) is optimal in the minimax sense.

\section{Simulation study}\label{a_section_simu1}

In this section, we study simulations based on the framework given by (\ref{a_frame}) with $K+1$ components $s,t^1,\dots,t^K$ and Gaussian errors. First, we introduce the spaces $\mathcal{S}_n$ and $\mathcal{S}_n^j,\ j\in\{1,\dots,K\}$, and the collections of models that we handle. Next, we illustrate the performances of the estimators in practice by several examples.

\subsection{Preliminaries}

To perform the simulation study, we consider two collections of models. In both cases, we deal with the same spaces $\mathcal{S}_n$ and $\mathcal{S}_n^j$ defined as follows. Let $\varphi$ be the Haar wavelet's mother function,
$$
\forall x\in\R,\ \varphi(x)=
\left\{
\begin{array}{cl}
1 & \text{if }0\leqslant x<1/2\ ,\\
-1 & \text{if }1/2\leqslant x<1\ ,\\
0 & \text{otherwise.}
\end{array}
\right.
$$
For any $i\in\N$ and $j\in\{0,\dots,2^i-1\}$, we introduce the functions
$$\varphi_{i,j}(x)=2^{i/2}\varphi(2^ix-j),\ x\in\R\ .$$
It is clear that these functions are orthonormal in $\Ltwo_0([0,1],dx)$ for the usual scalar product. Let $d_n$ be some positive integer, we consider the space $\mathcal{S}_n\subset\Ltwo_0([0,1],dx)$ generated by the functions $\varphi_{i,j}$ such that $0\leqslant i\leqslant d_n$ and $0\leqslant j<2^i$. The dimension of this space is $\dim(\mathcal{S}_n)=D_n=2^{d_n+1}-1$. In the sequel, we denote by $\Pi_n$ the set of all the allowed pairs $(i,j)$,
$$\Pi_n=\left\{(i,j)\in\N^2\text{ such that }0\leqslant i\leqslant d_n,\ 0\leqslant j<2^i\right\}\ .$$
Moreover, for any $k\in\{1,\dots,D_n\}$ such that $k=2^i+j$ with $(i,j)\in\Pi_n$, we denote $\phi_k=\varphi_{i,j}$.

Let $d_n'$ be an other positive integer, the spaces $\mathcal{S}_n^j\subset\Ltwo_0([0,1],dy^j)$ are all supposed to be generated by the functions defined on $[0,1]$ by
$$\begin{array}{lcr}
\psi_{2i}(y)=\psi_{2i}^{(j)}(y)=\sin(i\pi y) & \text{ and } & \psi_{2i-1}(y)=\psi_{2i-1}^{(j)}(y)=\cos(i\pi y)
\end{array}$$
for any $i\in\{1,\dots,d_n'\}$ and $j\in\{1,\dots,K\}$. Thus, we have $\dim(\mathcal{S}_n^j)=D_n^{(j)}=2d_n'$ and $D_n'=2Kd_n'+1$.

As previously, we define $P_n$ as the oblique projector onto $E$ along $F+(E+F)^{\perp}$. The image set $E=\im(P_n)$ is generated by the vectors
$$\varphi_{i,j}=(\varphi_{i,j}(x_1),\dots,\varphi_{i,j}(x_n))'\in\R^n,\ (i,j)\in\Pi_n\ .$$
Let $m$ be a subset of $\Pi_n$, the model $S_m$ is defined as the linear subspace of $E$ generated by the vectors $\varphi_{i,j}$ with $(i,j)\in m$.

In the following simulations, we always take $D_n$ and $D_n'$ close to $4\sqrt{n}$, \textit{i.e.}
$$\begin{array}{lcr}
\displaystyle{d_n=\left\lfloor\frac{\log(2\sqrt{n}+1/2)}{\log(2)}\right\rfloor} & \text{and} & \displaystyle{d_n'=\left\lfloor\frac{4\sqrt{n}-1}{2K}\right\rfloor}
\end{array}$$
where, for any $x\in\R$, $\lfloor x\rfloor$ denotes the largest integer not greater than $x$. For such choices, basic computations lead to $L_n$ of order of $n^{5/4}$ in Proposition \ref{prop_rho}. As a consequence, this proposition does not ensure that {\bf ($\text{A}_3$)} is fulfilled with a large probability. However, $\rho(P_n)$ remains small in practice as we will see and it allows us to deal with larger collections of models.

\subsection{Collections of models}

The first collection of models is the smaller one because the models are nested. Let us introduce the index subsets, for any $i\in\{0,\dots,d_n\}$,
$$m_i=\left\{(i,j),\ 0\leqslant j<2^i\right\}\subset\Pi_n\ .$$
Thus, we define $\mathcal{F}^N$ as
$$\mathcal{F}^N=\left\{S_m\text{ such that }\exists k\in\{0,\dots,d_n\},\ m=\bigcup_{i=0}^km_i\right\}\ .$$
This collection has a small complexity since, for any $d\in\N$, $N_d\leqslant 1$. According to Corollary \ref{G_gen_coro1}, we can consider the penalty function given by
\begin{equation}\label{Npen1}
\pen_N(m)=(1+C)\frac{\Tr(\tra{P_n}\pi_mP_n)}{n}\sigma^2
\end{equation}
for some $C>0$. In order to compute the selected estimator $\tilde{s}$, we simply compute $\hat{s}_m$ in each model of $\mathcal{F}^N$ and we take the one that minimizes the penalized least-squares criterion.

The second collection of models is larger than $\mathcal{F}^N$. Indeed, we allow $m$ to be any subset of $\Pi_n$ and we introduce
$$\mathcal{F}^C=\left\{S_m\text{ such that }m\subset\Pi_n\right\}\ .$$
The complexity of this collection is large because, for any $d\in\N$,
$$N_d=\binom{D_n}{d}=\frac{D_n!}{d!(D_n-d)!}\leqslant\left(\frac{eD_n}{d}\right)^d\ .$$
So, we have $\log N_d\leqslant d(1+\log D_n)$ and, according to Corollary \ref{G_gen_coro1}, we take a penalty function as
\begin{equation}\label{Cpen1}
\pen_C(m)=(1+C+\log D_n)\frac{\Tr(\tra{P_n}\pi_mP_n)}{n}\sigma^2
\end{equation}
for some $C>0$. The large number of models in $\mathcal{F}^C$ leads to difficulties for computing the estimator $\tilde{s}$. Instead of exploring all the models among $\mathcal{F}^C$, we break the penalized criterion down with respect to an orthonormal basis $\phi_1,\dots,\phi_{D_n}$ of $E$ and we get
\begin{eqnarray*}
& & \left\Vert Y-\sum_{i=1}^{D_n}\langle Y,\phi_i\rangle_n\phi_i\right\Vert_n^2+(1+C+\log D_n)\frac{\Tr(\tra{P_n}\pi_EP_n)}{n}\sigma^2\\
& = & \Vert Y\Vert_n^2-\sum_{i=1}^{D_n}\left[\langle Y,\phi_i\rangle_n^2-(1+C+\log D_n)\Vert\tra{P_n}\phi_i\Vert_n^2\sigma^2\right]\ .
\end{eqnarray*}
In order to minimize the penalized least-squares criterion, we only need to keep the coefficients $\langle Y,\phi_i\rangle_n$ that are such that
$$\langle Y,\phi_i\rangle_n^2\geqslant(1+C+\log D_n)\Vert\tra{P_n}\phi_i\Vert_n^2\sigma^2\ .$$
This threshold procedure allows us to compute the estimator $\tilde{s}$ in reasonable time.

In accordance with the results of Section \ref{a_unknown_var}, in the case of unknown variance, we substitute $\hat{\sigma}^2$ for $\sigma^2$ in the penalties (\ref{Npen1}) and (\ref{Cpen1}).

\subsection{Numerical simulations}

We now illustrate our results and the performances of our estimation procedure by applying it to simulated data
$$Z_i=s(x_i)+\sum_{j=1}^Kt^j(y^j_i)+\sigma\varepsilon_i,\ i=1,\dots,n\ ,$$
where $K\geqslant 1$ is an integer that will vary from an experiment to an other, the design points $(x_i,y^1_i,\dots,y^K_i)'$ are known independent realizations of an uniform random variable on $[0,1]^{K+1}$ and the errors $\varepsilon_i$ are i.i.d. standard Gaussian random variables. We handle this framework with known or unknown variance factor $\sigma^2=1$ according to the cases and we consider a design of size $n=512$. The unknown components $s,t^1,\dots,t^K$ are either chosen among the following ones, or set to zero in the last subsection,
$$\begin{array}{lll}
\displaystyle{f_1(x)=\sin\left(4\pi\left(x\wedge\frac{1}{2}\right)\right)} & \displaystyle{f_2(x)=\cos\left(2\pi\left(x-\frac{1}{4}\right)^2\right)}-C_2 & f_3(x)=x+2\exp(-16x^2)-C_3\\
~ & ~ & ~\\
f_4(x)=\sin(2x)+2\exp(-16x^2)-C_4 & \displaystyle{f_5(x)=\frac{1-\exp(-10(x-1/2))}{1+\exp(-10(x-1/2))}} & f_6(x)=6x(1-x)-1
\end{array}$$
where the constants $C_2$, $C_3$ and $C_4$ are such that $f_i\in\Ltwo_0([0,1],dx)$ for any $i\in\{1,\dots,6\}$.

The first step of the procedure consists in computing the oblique projector $P_n$ and taking the data $Y=P_nZ$. Figure \ref{fig:simu01} gives an example by representing the signal $s$, the data $Z$ and the projected data $Y$ for $K=6$, $s=f_1$ and $t^j=f_j$, $j\in\{1,\dots,6\}$. In particular, for this example, we have $\rho^2(P_n)=1.22$. We see that we actually get reasonable value of $\rho^2(P_n)$ with our particular choices for $D_n$ and $D_n'$.

\begin{figure}[htb]
\centering
\includegraphics[width=0.5\textwidth]{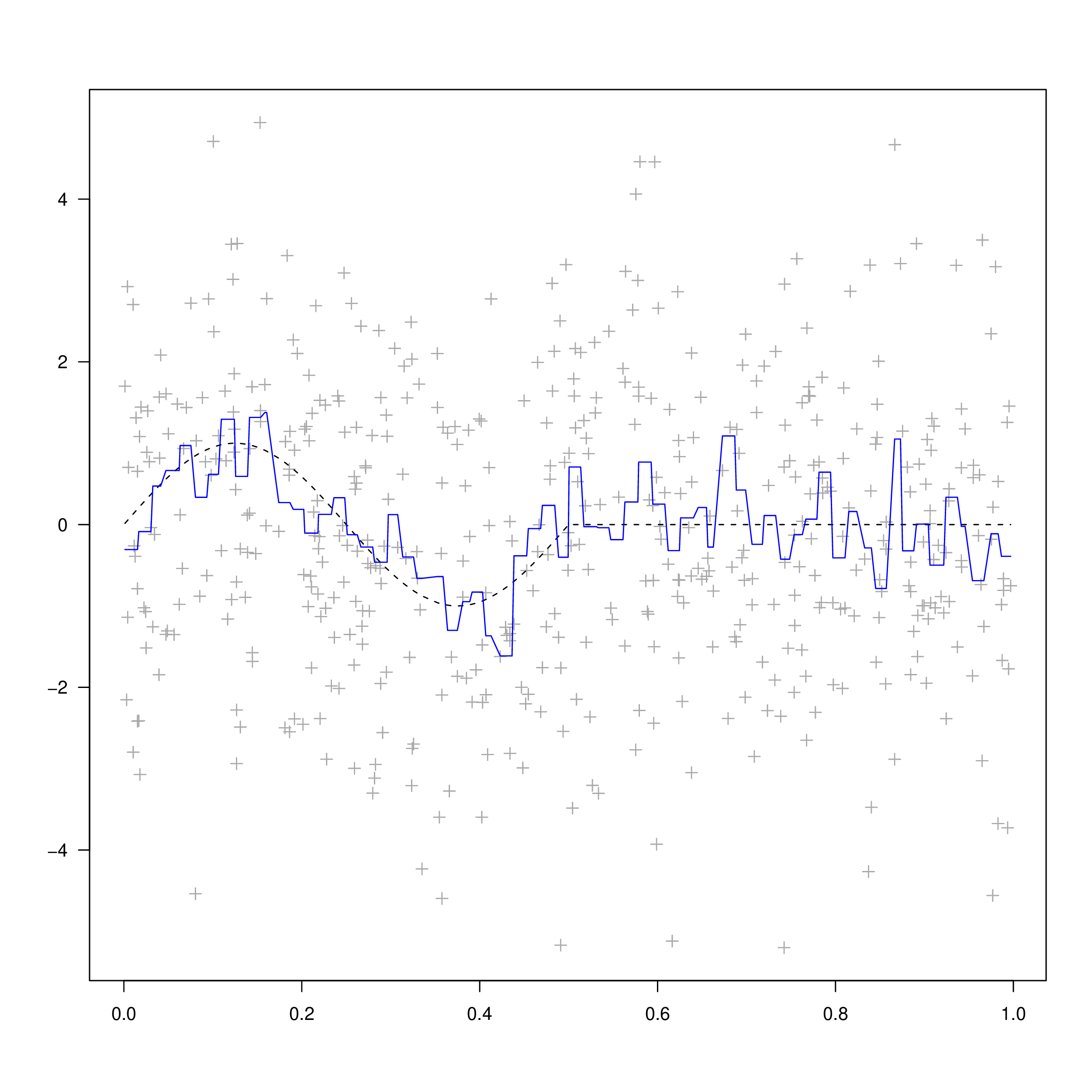}
\caption{Plot in $(x,z)$ of the signal $s$ (dashed line), the data $Z$ (dots) and the projected data $Y$ (plain line).}
\label{fig:simu01}
\end{figure}

In order to estimate the component $s$, we choose $\hat{m}$ by the procedure \eqref{pen_crit} with penalty function given by \eqref{Npen1} or \eqref{Cpen1} according to the cases. The first simulations deal with the collection $\mathcal{F}^N$ of nested models. Figure \ref{fig:simu02_03} represents the true $s$ and the estimator $\tilde{s}$ for $K=6$ parasitic components given by $t^j=f_j,\ j\in\{1,\dots,6\}$ and $s=f_1$ or $s=f_5$. The penalty function \eqref{Npen1} has been used with a constant $C=1.5$.

\begin{figure}[htb]
\includegraphics[width=0.4\textwidth]{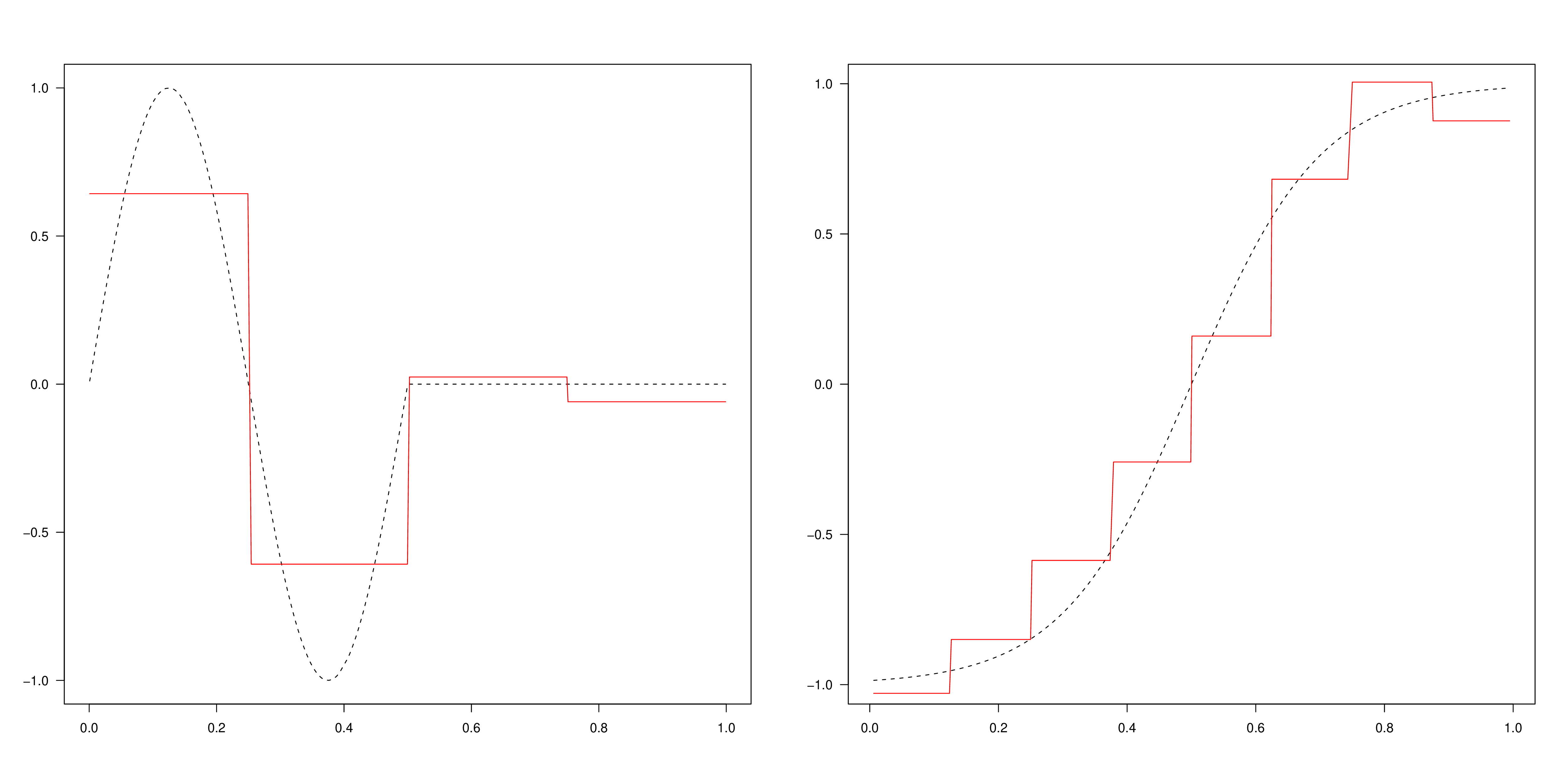}
\caption{Estimation of $s$ (dashed) by $\tilde{s}$ (plain) with $\mathcal{F}^N$, $K=6$ and $t^j=f_j$, $j\in\{1,\dots,6\}$, for $s=f_1$ (left, $\rho(P_n)=1.24$) and for $s=f_5$ (right, $\rho(P_n)=1.25$).}
\label{fig:simu02_03}
\end{figure}

The second set of simulations is related to the large collection $\mathcal{F}^C$ and to the penalty function \eqref{Cpen1} with $C=4.5$. Figure \ref{fig:simu04_05} illustrates the estimation of $s=f_1$ and $s=f_2$ with $K=6$ parasitic components $t^j=f_j,\ j\in\{1,\dots,6\}$.

\begin{figure}[htb]
\includegraphics[width=0.4\textwidth]{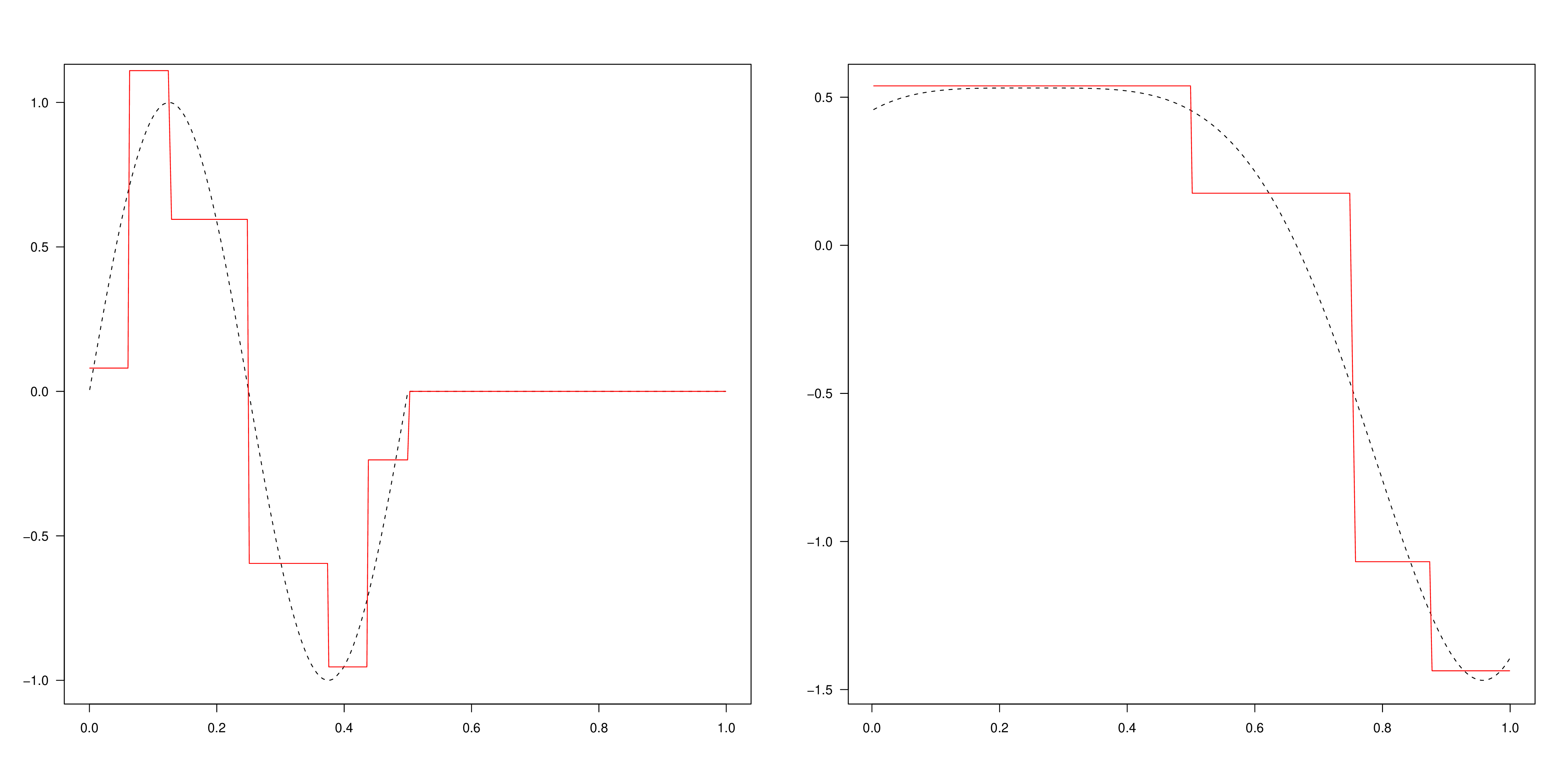}
\caption{Estimation of $s$ (dashed) by $\tilde{s}$ (plain) with $\mathcal{F}^C$, $K=6$ and $t^j=f_j$, $j\in\{1,\dots,6\}$, for $s=f_1$ (left, $\rho(P_n)=1.23$) and for $s=f_2$ (right, $\rho(P_n)=1.27$).}
\label{fig:simu04_05}
\end{figure}

In both cases, we see that the estimation procedure behaves well and that the norms $\rho(P_n)$ are close to one in spite of the presence of the parasitic components. Moreover, note that the collection $\mathcal{F}^C$ allows to get estimators that are sharper because they detect constant parts of $s$. This advantage leads to a better bias term in the quadratic risk decomposition at the price of the logarithmic term in the penalty \eqref{Cpen1}.

\subsection{Ratio estimation}

In Section \ref{sect_estim_comp}, we discussed about assumptions that ensure a small remainder term in Inequality \eqref{a_ineg02}. This result corresponds to some oracle type inequality for our estimation procedure of a component in an additive framework. We want to evaluate how far $\E\left[\Vert s-\tilde{s}\Vert_n^2\right]$ is from the oracle risk. Thus, we estimate the ratio
$$r_K(\tilde{s})=\frac{\E\left[\Vert s-\tilde{s}\Vert_n^2\right]}{\displaystyle{\inf_{m\in\mathcal{M}}\left\{\Vert s-s_m\Vert_n^2+\frac{\Tr(\tra{P_n}\pi_mP_n)}{n}\sigma^2\right\}}}$$
by repeating $500$ times each experiment for various values of $K$ and $C$. For each set of simulations, the parasitic components are taken such that $t^j=f_j,\ j\in\{1,\dots,K\}$, the values of $\rho(P_n)$ are given and the variance $\sigma^2$ is either assumed to be known or not.

Table \ref{table:simu02_ratio} (resp. Table \ref{table:simu03_ratio}) gives the values of $r_K(\tilde{s})$ obtained for $s=f_1$ (resp. $s=f_5$) with the collection $\mathcal{F}^N$ and the penalty \eqref{Npen1}. We clearly see that taking $C$ close to zero or too large is not a good thing for the procedure. In our examples, $C=1.5$ give good results and we get reasonable values of $r_K(\tilde{s})$ for other choices of $C$ between 1 and 3 for known or unknown variance. As expected, we also note that the values of $\rho(P_n)$ and $r_K(\tilde{s})$ tend to increase when $K$ goes up but remain acceptable for $K\in\{1,\dots,6\}$.

\begin{table}[htb]
\centering
\begin{tabular}{|*{12}{c|}}
\hline
$C$ & 0.0 & 0.5 & 1.0 & 1.5 & 2.0 & 2.5 & 3.0 & 3.5 & 4.0 & 4.5 & 5.0\\
\hline
\multirow{2}{2.5cm}{$K=1$, $\rho(P_n)=1.23$}
 & 2.41 & 1.36 & 1.15 & 1.13 & 1.11 & 1.10 & 1.09 & 1.08 & 1.08 & 1.08 & 1.08 \\  \cline{2-12}
 & 1.46 & 1.29 & 1.19 & 1.14 & 1.10 & 1.09 & 1.09 & 1.09 & 1.08 & 1.08 & 1.08 \\
\hline
\multirow{2}{2.5cm}{$K=2$, $\rho(P_n)=1.23$}
 & 2.47 & 1.37 & 1.16 & 1.14 & 1.13 & 1.12 & 1.11 & 1.09 & 1.09 & 1.09 & 1.09 \\  \cline{2-12}
 & 1.55 & 1.26 & 1.18 & 1.14 & 1.12 & 1.12 & 1.11 & 1.10 & 1.09 & 1.09 & 1.09 \\
\hline
\multirow{2}{2.5cm}{$K=3$, $\rho(P_n)=1.28$}
 & 2.48 & 1.39 & 1.15 & 1.13 & 1.12 & 1.10 & 1.09 & 1.08 & 1.08 & 1.08 & 1.08 \\  \cline{2-12}
 & 2.34 & 1.26 & 1.16 & 1.13 & 1.11 & 1.10 & 1.09 & 1.09 & 1.08 & 1.08 & 1.08 \\
\hline
\multirow{2}{2.5cm}{$K=4$, $\rho(P_n)=1.25$}
 & 2.65 & 1.41 & 1.17 & 1.14 & 1.13 & 1.11 & 1.09 & 1.08 & 1.08 & 1.08 & 1.08 \\  \cline{2-12}
 & 1.46 & 1.27 & 1.16 & 1.13 & 1.11 & 1.10 & 1.09 & 1.09 & 1.08 & 1.08 & 1.08 \\
\hline
\multirow{2}{2.5cm}{$K=5$, $\rho(P_n)=1.29$}
 & 2.97 & 1.62 & 1.27 & 1.19 & 1.15 & 1.12 & 1.10 & 1.09 & 1.08 & 1.07 & 1.07 \\  \cline{2-12}
 & 1.63 & 1.38 & 1.26 & 1.19 & 1.13 & 1.11 & 1.09 & 1.08 & 1.08 & 1.08 & 1.07 \\
\hline
\multirow{2}{2.5cm}{$K=6$, $\rho(P_n)=1.27$}
 & 3.14 & 1.77 & 1.29 & 1.21 & 1.17 & 1.13 & 1.12 & 1.10 & 1.10 & 1.09 & 1.09 \\  \cline{2-12}
 & 1.66 & 1.40 & 1.26 & 1.18 & 1.14 & 1.13 & 1.11 & 1.11 & 1.10 & 1.10 & 1.09 \\
\hline
\end{tabular}
\vspace{\baselineskip}
\caption{\label{table:simu02_ratio}Ratio $r_K(\tilde{s})$ for the estimation of $s=f_1$ with $\mathcal{F}^N$. Each pair of lines corresponds to a value of $K$ with the known $\sigma^2$ case on the first line and unknown $\sigma^2$ case on the second one.}
\end{table}

\begin{table}[htb]
\centering
\begin{tabular}{|*{12}{c|}}
\hline
$C$ & 0.0 & 0.5 & 1.0 & 1.5 & 2.0 & 2.5 & 3.0 & 3.5 & 4.0 & 4.5 & 5.0\\
\hline
\multirow{2}{2.5cm}{$K=1$, $\rho(P_n)=1.28$}
 & 4.08 & 1.52 & 1.22 & 1.20 & 1.27 & 1.35 & 1.45 & 1.56 & 1.64 & 1.70 & 1.79 \\  \cline{2-12}
 & 3.44 & 1.58 & 1.36 & 1.26 & 1.30 & 1.37 & 1.45 & 1.55 & 1.64 & 1.72 & 1.81 \\
\hline
\multirow{2}{2.5cm}{$K=2$, $\rho(P_n)=1.23$}
 & 4.07 & 1.66 & 1.28 & 1.26 & 1.32 & 1.40 & 1.49 & 1.57 & 1.66 & 1.74 & 1.82 \\  \cline{2-12}
 & 2.29 & 1.69 & 1.36 & 1.32 & 1.36 & 1.44 & 1.53 & 1.60 & 1.65 & 1.73 & 1.82 \\
\hline
\multirow{2}{2.5cm}{$K=3$, $\rho(P_n)=1.25$}
 & 4.17 & 1.65 & 1.36 & 1.34 & 1.42 & 1.50 & 1.60 & 1.67 & 1.77 & 1.89 & 2.01 \\  \cline{2-12}
 & 2.24 & 1.70 & 1.41 & 1.41 & 1.48 & 1.55 & 1.61 & 1.71 & 1.80 & 1.92 & 2.01 \\
\hline
\multirow{2}{2.5cm}{$K=4$, $\rho(P_n)=1.26$}
 & 4.42 & 1.88 & 1.43 & 1.34 & 1.36 & 1.45 & 1.53 & 1.61 & 1.69 & 1.77 & 1.86 \\  \cline{2-12}
 & 3.80 & 1.75 & 1.51 & 1.42 & 1.44 & 1.50 & 1.56 & 1.66 & 1.75 & 1.84 & 1.93 \\
\hline
\multirow{2}{2.5cm}{$K=5$, $\rho(P_n)=1.26$}
 & 4.57 & 1.82 & 1.43 & 1.37 & 1.39 & 1.46 & 1.53 & 1.60 & 1.67 & 1.76 & 1.83 \\  \cline{2-12}
 & 2.33 & 1.77 & 1.51 & 1.43 & 1.44 & 1.50 & 1.54 & 1.64 & 1.74 & 1.82 & 1.89 \\
\hline
\multirow{2}{2.5cm}{$K=6$, $\rho(P_n)=1.27$}
 & 4.98 & 2.08 & 1.59 & 1.47 & 1.45 & 1.49 & 1.57 & 1.66 & 1.77 & 1.86 & 1.96 \\  \cline{2-12}
 & 2.57 & 1.91 & 1.62 & 1.52 & 1.54 & 1.57 & 1.65 & 1.73 & 1.84 & 1.93 & 2.02 \\
\hline
\end{tabular}
\vspace{\baselineskip}
\caption{\label{table:simu03_ratio}Ratio $r_K(\tilde{s})$ for the estimation of $s=f_5$ with $\mathcal{F}^N$. Each pair of lines corresponds to a value of $K$ with the known $\sigma^2$ case on the first line and unknown $\sigma^2$ case on the second one.}
\end{table}

In the same way, we estimate the ratio $r_K(\tilde{s})$ for $s=f_1$ and $s=f_2$ with the collection $\mathcal{F}^C$ and the penalty \eqref{Cpen1}. The results are given in Table \ref{table:simu04_ratio} and Table \ref{table:simu05_ratio}. We obtain reasonable values of $r_K(\tilde{s})$ for choices of $C$ larger than what we took in the nested case. This phenomenon is related to what we mentioned at the end of Section \ref{a_section_known}. Indeed, for large collection of models, we need to overpenalize in order to keep the remainder term small enough. Moreover, because $\hat{\sigma}^2$ tends to overestimate $\sigma^2$ (see Section \ref{a_unknown_var}), we see that we can consider smaller values for $C$ when the variance is unknown than when it is known for obtaining equivalent results.

\begin{table}[htb]
\centering
\begin{tabular}{|*{12}{c|}}
\hline
$C$ & 0.0 & 0.5 & 1.0 & 1.5 & 2.0 & 2.5 & 3.0 & 3.5 & 4.0 & 4.5 & 5.0\\
\hline
\multirow{2}{2.5cm}{$K=1$, $\rho(P_n)=1.27$}
 & 1.54 & 1.49 & 1.44 & 1.40 & 1.36 & 1.33 & 1.31 & 1.30 & 1.28 & 1.27 & 1.25 \\  \cline{2-12}
 & 1.50 & 1.44 & 1.39 & 1.35 & 1.32 & 1.30 & 1.28 & 1.26 & 1.25 & 1.24 & 1.23 \\
\hline
\multirow{2}{2.5cm}{$K=2$, $\rho(P_n)=1.25$}
 & 1.60 & 1.53 & 1.48 & 1.45 & 1.40 & 1.37 & 1.34 & 1.32 & 1.29 & 1.28 & 1.26 \\  \cline{2-12}
 & 1.54 & 1.48 & 1.42 & 1.38 & 1.35 & 1.32 & 1.29 & 1.28 & 1.27 & 1.25 & 1.24 \\
\hline
\multirow{2}{2.5cm}{$K=3$, $\rho(P_n)=1.25$}
 & 1.56 & 1.50 & 1.46 & 1.42 & 1.38 & 1.35 & 1.32 & 1.30 & 1.28 & 1.27 & 1.26 \\  \cline{2-12}
 & 1.51 & 1.45 & 1.41 & 1.37 & 1.34 & 1.31 & 1.29 & 1.27 & 1.25 & 1.24 & 1.23 \\
\hline
\multirow{2}{2.5cm}{$K=4$, $\rho(P_n)=1.25$}
 & 1.61 & 1.54 & 1.48 & 1.42 & 1.39 & 1.36 & 1.34 & 1.31 & 1.29 & 1.28 & 1.27 \\  \cline{2-12}
 & 1.51 & 1.44 & 1.40 & 1.36 & 1.32 & 1.31 & 1.28 & 1.27 & 1.26 & 1.25 & 1.24 \\
\hline
\multirow{2}{2.5cm}{$K=5$, $\rho(P_n)=1.25$}
 & 1.68 & 1.61 & 1.54 & 1.48 & 1.44 & 1.41 & 1.37 & 1.34 & 1.32 & 1.30 & 1.28 \\  \cline{2-12}
 & 1.56 & 1.49 & 1.43 & 1.39 & 1.36 & 1.31 & 1.29 & 1.27 & 1.27 & 1.26 & 1.25 \\
\hline
\multirow{2}{2.5cm}{$K=6$, $\rho(P_n)=1.24$}
 & 1.78 & 1.70 & 1.63 & 1.57 & 1.53 & 1.48 & 1.44 & 1.42 & 1.39 & 1.35 & 1.34 \\  \cline{2-12}
 & 1.61 & 1.55 & 1.48 & 1.44 & 1.40 & 1.37 & 1.34 & 1.32 & 1.30 & 1.28 & 1.28 \\
\hline
\end{tabular}
\vspace{\baselineskip}
\caption{\label{table:simu04_ratio}Ratio $r_K(\tilde{s})$ for the estimation of $s=f_1$ with $\mathcal{F}^C$. Each pair of lines corresponds to a value of $K$ with the known $\sigma^2$ case on the first line and unknown $\sigma^2$ case on the second one.}
\end{table}

\begin{table}[htb]
\centering
\begin{tabular}{|*{12}{c|}}
\hline
$C$ & 0.0 & 0.5 & 1.0 & 1.5 & 2.0 & 2.5 & 3.0 & 3.5 & 4.0 & 4.5 & 5.0\\
\hline
\multirow{2}{2.5cm}{$K=1$, $\rho(P_n)=1.28$}
 & 2.01 & 1.92 & 1.86 & 1.80 & 1.76 & 1.74 & 1.70 & 1.70 & 1.68 & 1.67 & 1.68 \\  \cline{2-12}
 & 2.03 & 1.93 & 1.87 & 1.81 & 1.77 & 1.72 & 1.68 & 1.65 & 1.65 & 1.66 & 1.67 \\
\hline
\multirow{2}{2.5cm}{$K=2$, $\rho(P_n)=1.22$}
 & 2.02 & 1.93 & 1.85 & 1.79 & 1.75 & 1.71 & 1.68 & 1.66 & 1.66 & 1.66 & 1.66 \\  \cline{2-12}
 & 1.95 & 1.88 & 1.82 & 1.78 & 1.75 & 1.71 & 1.68 & 1.67 & 1.65 & 1.64 & 1.64 \\
\hline
\multirow{2}{2.5cm}{$K=3$, $\rho(P_n)=1.26$}
 & 2.04 & 1.93 & 1.86 & 1.81 & 1.76 & 1.71 & 1.68 & 1.64 & 1.62 & 1.62 & 1.62 \\  \cline{2-12}
 & 1.96 & 1.87 & 1.80 & 1.74 & 1.68 & 1.66 & 1.63 & 1.63 & 1.61 & 1.62 & 1.62 \\
\hline
\multirow{2}{2.5cm}{$K=4$, $\rho(P_n)=1.25$}
 & 2.12 & 2.00 & 1.90 & 1.81 & 1.73 & 1.67 & 1.64 & 1.62 & 1.60 & 1.61 & 1.60 \\  \cline{2-12}
 & 1.99 & 1.90 & 1.80 & 1.73 & 1.68 & 1.65 & 1.62 & 1.60 & 1.60 & 1.60 & 1.60 \\
\hline
\multirow{2}{2.5cm}{$K=5$, $\rho(P_n)=1.24$}
 & 2.47 & 2.34 & 2.23 & 2.17 & 2.10 & 2.05 & 1.99 & 1.95 & 1.91 & 1.88 & 1.86 \\  \cline{2-12}
 & 2.30 & 2.20 & 2.11 & 2.03 & 1.97 & 1.92 & 1.88 & 1.83 & 1.82 & 1.80 & 1.80 \\
\hline
\multirow{2}{2.5cm}{$K=6$, $\rho(P_n)=1.26$}
 & 2.45 & 2.32 & 2.21 & 2.11 & 2.03 & 1.99 & 1.95 & 1.91 & 1.89 & 1.86 & 1.84 \\  \cline{2-12}
 & 2.17 & 2.06 & 1.99 & 1.94 & 1.89 & 1.85 & 1.84 & 1.80 & 1.79 & 1.79 & 1.75 \\
\hline
\end{tabular}
\vspace{\baselineskip}
\caption{\label{table:simu05_ratio}Ratio $r_K(\tilde{s})$ for the estimation of $s=f_2$ with $\mathcal{F}^C$. Each pair of lines corresponds to a value of $K$ with the known $\sigma^2$ case on the first line and unknown $\sigma^2$ case on the second one.}
\end{table}

\subsection{Parasitic components equal to zero}

We are now interested in the particular case of parasitic components $t^j$ equal to zero in \eqref{a_frame}, \text{i.e.} data are given by
$$Z_i=s(x_i)+\sigma\varepsilon_i,\ i=1,\dots,n\ .$$
If we know that these $K$ components are zero and if we deal with the collection $\mathcal{F}^N$ and a known variance $\sigma^2$, we can consider the classical model selection procedure given by
\begin{equation}\label{bm_proc}
\hat{m}_0\in\argmin_{m\in\mathcal{M}}\left\{\Vert Z-\pi_mZ\Vert_n^2+C\frac{\dim(S_m)}{n}\sigma^2\right\}\ .
\end{equation}
Then, we can define the estimator $\tilde{s}_0=\pi_{\hat{m}_0}Z$. This procedure is well known and we refer to \cite{Mas07} for more details. If we do not know that the $K$ parasitic components are null, we can use our procedure to estimate $s$ by $\tilde{s}$. In order to compare the performances of $\tilde{s}$ and $\tilde{s}_0$ with respect to the number $K$ of zero parasitic components, we estimate the ratio
$$r_K(\tilde{s},\tilde{s}_0)=\frac{\E[\Vert s-\tilde{s}\Vert_n^2]}{\E[\Vert s-\tilde{s}_0\Vert_n^2]}$$
for various values of $K$ and $C$ by repeating 500 times each experiment.

The obtained results are given in Tables \ref{table:simu06_ratio} and \ref{table:simu07_ratio} for $s=f_1$ and $s=f_5$ respectively. Obviously, the ratio $r_K(\tilde{s},\tilde{s}_0)$ is always larger than one because the procedure \eqref{bm_proc} makes good use of its knowledge about nullity of the $t^j$. Nevertheless, we see that our procedure performs nearly as well as \eqref{bm_proc} even for a large number of zero components. Indeed, for $K\in\{1,\dots,9\}$, do not assuming that we know that the $t^j$ are zero only implies a loss between $1\%$ and $10\%$ for the risk. Such a loss remains acceptable in practice and allows us to consider more general framework for estimating $s$.

\begin{table}[htb]
\centering
\begin{tabular}{|*{12}{c|}}
\hline
$C$ & 0.0 & 0.5 & 1.0 & 1.5 & 2.0 & 2.5 & 3.0 & 3.5 & 4.0 & 4.5 & 5.0\\
\hline
$K=1$
 & 1.11 & 1.11 & 1.09 & 1.06 & 1.04 & 1.03 & 1.03 & 1.02 & 1.01 & 1.02 & 1.02 \\
\hline
$K=2$
 & 1.12 & 1.08 & 1.08 & 1.06 & 1.04 & 1.03 & 1.02 & 1.01 & 1.01 & 1.01 & 1.01 \\
\hline
$K=3$
 & 1.13 & 1.09 & 1.07 & 1.07 & 1.05 & 1.03 & 1.01 & 1.01 & 1.02 & 1.02 & 1.02 \\
\hline
$K=4$
 & 1.08 & 1.08 & 1.06 & 1.05 & 1.04 & 1.02 & 1.02 & 1.01 & 1.01 & 1.01 & 1.01 \\
\hline
$K=5$
 & 1.10 & 1.05 & 1.06 & 1.06 & 1.03 & 1.02 & 1.02 & 1.01 & 1.01 & 1.01 & 1.01 \\
\hline
$K=6$
 & 1.08 & 1.07 & 1.06 & 1.05 & 1.03 & 1.02 & 1.01 & 1.01 & 1.01 & 1.01 & 1.01 \\
\hline
$K=7$
 & 1.11 & 1.09 & 1.08 & 1.05 & 1.03 & 1.02 & 1.01 & 1.01 & 1.01 & 1.01 & 1.01 \\
\hline
$K=8$
 & 1.09 & 1.06 & 1.08 & 1.05 & 1.04 & 1.02 & 1.01 & 1.01 & 1.01 & 1.01 & 1.01 \\
\hline
$K=9$
 & 1.10 & 1.08 & 1.07 & 1.05 & 1.03 & 1.02 & 1.01 & 1.01 & 1.01 & 1.01 & 1.01 \\
\hline
\end{tabular}
\vspace{\baselineskip}
\caption{\label{table:simu06_ratio}Ratio $r_K(\tilde{s},\tilde{s}_0)$ for the estimation of $s=f_1$ with $\mathcal{F}^N$.}
\end{table}

\begin{table}[htb]
\centering
\begin{tabular}{|*{12}{c|}}
\hline
$C$ & 0.0 & 0.5 & 1.0 & 1.5 & 2.0 & 2.5 & 3.0 & 3.5 & 4.0 & 4.5 & 5.0\\
\hline
$K=1$
 & 1.08 & 1.09 & 1.07 & 1.07 & 1.09 & 1.09 & 1.08 & 1.07 & 1.06 & 1.09 & 1.07 \\
\hline
$K=2$
 & 1.09 & 1.05 & 1.08 & 1.09 & 1.09 & 1.08 & 1.08 & 1.08 & 1.06 & 1.06 & 1.05 \\
\hline
$K=3$
 & 1.12 & 1.12 & 1.11 & 1.07 & 1.09 & 1.10 & 1.09 & 1.08 & 1.07 & 1.06 & 1.07 \\
\hline
$K=4$
 & 1.09 & 1.11 & 1.08 & 1.10 & 1.10 & 1.09 & 1.07 & 1.06 & 1.07 & 1.06 & 1.07 \\
\hline
$K=5$
 & 1.10 & 1.08 & 1.09 & 1.09 & 1.09 & 1.06 & 1.06 & 1.06 & 1.07 & 1.07 & 1.05 \\
\hline
$K=6$
 & 1.08 & 1.04 & 1.06 & 1.07 & 1.08 & 1.07 & 1.07 & 1.09 & 1.06 & 1.06 & 1.06 \\
\hline
$K=7$
 & 1.06 & 1.05 & 1.07 & 1.08 & 1.10 & 1.09 & 1.07 & 1.09 & 1.08 & 1.07 & 1.06 \\
\hline
$K=8$
 & 1.08 & 1.13 & 1.08 & 1.09 & 1.09 & 1.08 & 1.06 & 1.07 & 1.07 & 1.06 & 1.06 \\
\hline
$K=9$
 & 1.13 & 1.05 & 1.09 & 1.09 & 1.07 & 1.07 & 1.07 & 1.06 & 1.07 & 1.07 & 1.06 \\
\hline
\end{tabular}
\vspace{\baselineskip}
\caption{\label{table:simu07_ratio}Ratio $r_K(\tilde{s},\tilde{s}_0)$ for the estimation of $s=f_5$ with $\mathcal{F}^N$.}
\end{table}

\section{Proofs}

In the proofs, we repeatedly use the following elementary inequality that holds for any $\alpha>0$ and $x,y\in\R$,
\begin{equation}\label{inegCarres}
2\vert xy\vert\leqslant\alpha x^2+\alpha^{-1}y^2\ .
\end{equation}

\subsection{Proofs of Theorems \ref{G_main_theo} and \ref{NG_th1}}

\subsubsection{Proof of Theorem \ref{G_main_theo}}

By definition of $\gamma_n$, for any $t\in\R^n$, we can write
$$\Vert s-t\Vert_n^2=\gamma_n(t)+2\sigma\langle t-Y,P_n\varepsilon\rangle_n+\sigma^2\Vert P_n\varepsilon\Vert_n^2\ .$$
Let $m\in\M$, since $\hat{s}_m=s_m+\sigma\pi_mP_n\varepsilon$, this identity and (\ref{prop_def_plse}) lead to
\begin{eqnarray}
\Vert s-\tilde{s}\Vert_n^2 & = & \Vert s-s_m\Vert_n^2+\gamma_n(\tilde{s})-\gamma_n(s_m)+2\sigma\langle \tilde{s}-s_m,P_n\varepsilon\rangle_n\nonumber\\
& = & \Vert s-s_m\Vert_n^2+\gamma_n(\tilde{s})-\gamma_n(\hat{s}_m)-\sigma^2\Vert\pi_mP_n\varepsilon\Vert_n^2\nonumber\\
& & \hspace{0.5cm}-2\sigma\langle s-\tilde{s},P_n\varepsilon\rangle_n+2\sigma\langle s-s_m,P_n\varepsilon\rangle_n\nonumber\\
& \leqslant & \Vert s-s_m\Vert_n^2+\pen(m)-\pen(\hat{m})+2\sigma^2\Vert\pi_{\hat{m}}P_n\varepsilon\Vert_n^2\label{G_ineg1}\\
& & \hspace{0.5cm}-2\sigma\langle s-s_{\hat{m}},P_n\varepsilon\rangle_n+2\sigma\langle s-s_m,P_n\varepsilon\rangle_n-\sigma^2\Vert\pi_mP_n\varepsilon\Vert_n^2\ .\nonumber
\end{eqnarray}
Consider an arbitrary $a_m\in S_m^{\perp}$ such that $\Vert a_m\Vert_n=1$, we define
\begin{equation}\label{um_def}
u_m=\left\{
\begin{array}{ll}
(s-s_m)/\Vert s-s_m\Vert_n & \text{if }s\neq\pi_ms\\
a_m & \text{otherwise .}
\end{array}
\right.
\end{equation}
Thus, (\ref{G_ineg1}) gives
\begin{eqnarray}
\Vert s-\tilde{s}\Vert_n^2 & \leqslant & \Vert s-s_m\Vert_n^2+\pen(m)-\pen(\hat{m})+2\sigma^2\Vert\pi_{\hat{m}}P_n\varepsilon\Vert_n^2\\
& & +2\sigma\Vert s-s_{\hat{m}}\Vert_n\vert\langle u_{\hat{m}},P_n\varepsilon\rangle_n\vert+2\sigma\langle s-s_m,P_n\varepsilon\rangle_n-\sigma^2\Vert\pi_mP_n\varepsilon\Vert_n^2\ .\nonumber
\end{eqnarray}
Take $\alpha\in(0,1)$ that we specify later and we use the inequality (\ref{inegCarres}),
\begin{eqnarray}
(1-\alpha)\Vert s-\tilde{s}\Vert_n^2 & \leqslant & \Vert s-s_m\Vert_n^2+\pen(m)-\pen(\hat{m})\label{ineg01}\\
& & \hspace{0.5cm}+(2-\alpha)\sigma^2\Vert\pi_{\hat{m}}P_n\varepsilon\Vert_n^2+\alpha^{-1}\sigma^2\langle u_{\hat{m}},P_n\varepsilon\rangle_n^2\nonumber\\
& & \hspace{0.5cm}+2\sigma\langle s-s_m,P_n\varepsilon\rangle_n-\sigma^2\Vert\pi_mP_n\varepsilon\Vert_n^2\nonumber\ .
\end{eqnarray}
We choose $\alpha=1/(1+\theta)\in(0,1)$ but for legibility we keep using the notation $\alpha$. Let us now introduce two functions $p_1,p_2:\M\rightarrow\R_+$ that will be specified later to satisfy, for all $m\in\M$,
\begin{equation}\label{G_hyp_pen_alpha}
\pen(m)\geqslant(2-\alpha)p_1(m)+\alpha^{-1}p_2(m)\ .
\end{equation}

We use this bound in (\ref{ineg01}) to obtain
\begin{eqnarray*}
(1-\alpha)\Vert s-\tilde{s}\Vert_n^2 & \leqslant & \Vert s-s_m\Vert_n^2+\pen(m)+(2-\alpha)\left(\sigma^2\Vert\pi_{\hat{m}}P_n\varepsilon\Vert_n^2-p_1(\hat{m})\right)\\
& & \hspace{0.5cm}+\alpha^{-1}\left(\sigma^2\langle u_{\hat{m}},P_n\varepsilon\rangle_n^2-p_2(\hat{m})\right)+2\sigma\langle s-s_m,P_n\varepsilon\rangle_n\\
& & \hspace{0.5cm}-\sigma^2\Vert\pi_mP_n\varepsilon\Vert_n^2\\
& \leqslant & \Vert s-s_m\Vert_n^2+\pen(m)+2\sigma\langle s-s_m,P_n\varepsilon\rangle_n-\sigma^2\Vert\pi_mP_n\varepsilon\Vert_n^2\\
& & \hspace{0.5cm}+(2-\alpha)\sup_{m'\in\M}\left(\sigma^2\Vert\pi_{m'}P_n\varepsilon\Vert_n^2-p_1(m')\right)_+\\
& & \hspace{0.5cm}+\alpha^{-1}\sup_{m'\in\M}\left(\sigma^2\langle u_{m'},P_n\varepsilon\rangle_n^2-p_2(m')\right)_+\ .
\end{eqnarray*}
Taking the expectation on both sides, it leads to
\begin{eqnarray*}
(1-\alpha)\E\left[\Vert s-\tilde{s}\Vert_n^2\right] & \leqslant & \Vert s-s_m\Vert_n^2+\pen(m)-\Tr(\tra{P_n}\pi_mP_n)\sigma^2/n\\
& & \hspace{0.5cm}+(2-\alpha)\E\left[\sup_{m'\in\M}\left(\sigma^2\Vert\pi_{m'}P_n\varepsilon\Vert_n^2-p_1(m')\right)_+\right]\\
& & \hspace{0.5cm}+\alpha^{-1}\E\left[\sup_{m'\in\M}\left(\sigma^2\langle u_{m'},P_n\varepsilon\rangle_n^2-p_2(m')\right)_+\right]\\
& \leqslant & \Vert s-s_m\Vert_n^2+\pen(m)-\Tr(\tra{P_n}\pi_mP_n)\sigma^2/n\\
& & \hspace{0.5cm}+(2-\alpha)\sum_{m'\in\M}\E\left[\left(\sigma^2\Vert\pi_{m'}P_n\varepsilon\Vert_n^2-p_1(m')\right)_+\right]\\
& & \hspace{0.5cm}+\alpha^{-1}\sum_{m'\in\M}\E\left[\left(\sigma^2\langle u_{m'},P_n\varepsilon\rangle_n^2-p_2(m')\right)_+\right]\\
& \leqslant & \Vert s-s_m\Vert_n^2+\pen(m)-\Tr(\tra{P_n}\pi_mP_n)\sigma^2/n\\
& & \hspace{0.5cm}+(2-\alpha)\sum_{m'\in\M}\E_{1,m'}+\alpha^{-1}\sum_{m'\in\M}\E_{2,m'}\ .
\end{eqnarray*}
Because the choice of $m$ is arbitrary among $\M$, we can infer that
\begin{eqnarray}
(1-\alpha)\E\left[\Vert s-\tilde{s}\Vert_n^2\right] & \leqslant & \inf_{m\in\M}\left\{\Vert s-s_m\Vert_n^2+\pen(m)-\Tr(\tra{P_n}\pi_mP_n)\sigma^2/n\right\}\nonumber\\
& & \hspace{0.5cm}+(2-\alpha)\sum_{m\in\M}\E_{1,m}+\alpha^{-1}\sum_{m\in\M}\E_{2,m}\label{G_domin_E}\ .
\end{eqnarray}

We now have to upperbound $\E_{1,m}$ and $\E_{2,m}$ in (\ref{G_domin_E}). Let start by the first one. If $S_m=\{0\}$, then $\pi_mP_n=0$ and $p_1(m)\geqslant0$ suffices to ensure that $\E_{1,m}=0$. So, we can consider that the dimension of $S_m$ is positive and $\pi_mP_n\neq0$. The Lemma \ref{LM} applied with $A=\pi_mP_n$ gives, for any $x>0$,
\begin{equation}\label{inegLM01}
\Pg\left(n\Vert \pi_mP_n\varepsilon\Vert_n^2\geqslant\Tr(\tra{P_n}\pi_mP_n)+2\sqrt{\rho^2(P_n)\Tr(\tra{P_n}\pi_mP_n)x}+2\rho^2(P_n)x\right)\leqslant e^{-x}
\end{equation}
because $\rho(\pi_mP_n)\leqslant\rho(\pi_m)\rho(P_n)\leqslant\rho(P_n)$. Let $\beta=\theta^2/(1+2\theta)>0$, (\ref{inegCarres}) and (\ref{inegLM01}) lead to
\begin{equation}\label{inegLM02}
\Pg\left(n\Vert \pi_mP_n\varepsilon\Vert_n^2\geqslant(1+\beta)\Tr(\tra{P_n}\pi_mP_n)+(2+\beta^{-1})\rho^2(P_n)x\right)\leqslant e^{-x}\ .
\end{equation}
Let $\delta=\theta^2/((1+\theta)(1+2\theta+2\theta^2))>0$, we set
$$np_1(m)=((1+\beta)+(2+\beta^{-1})\delta L_m)\Tr(\tra{P_n}\pi_mP_n)\sigma^2$$
and (\ref{inegLM02}) implies
\begin{eqnarray}
\E_{m,1} & = & \int_0^{\infty}\Pg\left(\left(\sigma^2\Vert\pi_mP_n\varepsilon\Vert_n^2-p_1(m)\right)_+\geqslant\xi\right)d\xi\nonumber\\
& = & \int_0^{\infty}\Pg\left(n\Vert\pi_mP_n\varepsilon\Vert_n^2-np_1(m)/\sigma^2\geqslant n\xi/\sigma^2\right)d\xi\nonumber\\
& \leqslant & \int_0^{\infty}\exp\left(-\frac{\delta L_m\Tr(\tra{P_n}\pi_mP_n)}{\rho^2(P_n)}-\frac{n\xi}{(2+\beta^{-1})\rho^2(P_n)\sigma^2}\right)d\xi\nonumber\\
& \leqslant & \frac{(2+\beta^{-1})\rho^2(P_n)\sigma^2}{n}\exp\left(-\frac{\delta L_m\Tr(\tra{P_n}\pi_mP_n)}{\rho^2(P_n)}\right)\ .\label{inegGP1}
\end{eqnarray}

We now focus on $\E_{m,2}$. The random variable $\langle u_m,P_n\varepsilon\rangle_n=\langle \tra{P_n}u_m,\varepsilon\rangle_n$ is a centered Gaussian variable with variance $\Vert\tra{P_n}u_m\Vert_n^2/n$. For any $x>0$, the standard Gaussian deviation inequality gives
$$\Pg\left(\vert\langle u_m,P_n\varepsilon\rangle_n\vert\geqslant x\right)\leqslant \exp\left(-\frac{nx^2}{2\Vert\tra{P_n}u_m\Vert_n^2}\right)\leqslant\exp\left(-\frac{nx^2}{2\rho^2(P_n)}\right)$$
that is equivalent to
\begin{equation}\label{G_conc_classique}
\Pg\left(n\langle u_m,P_n\varepsilon\rangle_n^2\geqslant 2\rho^2(P_n)x\right)\leqslant e^{-x}\ .
\end{equation}
We set
$$np_2(m)=2\delta L_m\Tr(\tra{P_n}\pi_mP_n)\sigma^2$$
and (\ref{G_conc_classique}) leads to
\begin{eqnarray}
\E_{m,2} & = & \int_0^{\infty}\Pg\left(\left(\sigma^2\langle u_m,P_n\varepsilon\rangle_n^2-p_2(m)\right)_+\geqslant\xi\right)d\xi\nonumber\\
& = & \int_0^{\infty}\Pg\left(\langle u_m,P_n\varepsilon\rangle_n^2-np_2(m)/\sigma^2\geqslant n\xi/\sigma^2\right)d\xi\nonumber\\
& \leqslant & \int_0^{\infty}\exp\left(-\frac{\delta L_m\Tr(\tra{P_n}\pi_mP_n)}{\rho^2(P_n)}-\frac{n\xi}{2\rho^2(P_n)\sigma^2}\right)d\xi\nonumber\\
& \leqslant & \frac{2\rho^2(P_n)\sigma^2}{n}\exp\left(-\frac{\delta L_m\Tr(\tra{P_n}\pi_mP_n)}{\rho^2(P_n)}\right)\ .\label{inegGP2}
\end{eqnarray}

We inject (\ref{inegGP1}) and (\ref{inegGP2}) in (\ref{G_domin_E}) and we replace $\alpha$, $\beta$ and $\delta$ to obtain
\begin{equation*}
\frac{\theta}{\theta+1}\E\left[\Vert s-\tilde{s}\Vert_n^2\right] \leqslant \inf_{m\in\M}\left\{\Vert s-s_m\Vert_n^2+\pen(m)-\Tr(\tra{P_n}\pi_mP_n)\sigma^2/n\right\}+\frac{\rho^2(P_n)\sigma^2}{n}R_{\theta}
\end{equation*}
where we have set
\begin{equation*}
R_{\theta} = c_{\theta}\sum_{m\in\M}\exp\left(-\frac{L_m\Tr(\tra{P_n}\pi_mP_n)}{c_{\theta}\rho^2(P_n)}\right)
\end{equation*}
and
$$c_{\theta}=\frac{2\theta^4+8\theta^3+8\theta^2+4\theta+1}{\theta^2(1+\theta)}\ .$$
Finally, (\ref{G_hyp_pen_alpha}) gives a penalty as (\ref{G_pen_theo_gen}) and the announced result follows.

\subsubsection{Proof of Theorem \ref{NG_th1}}

In order to prove Theorem \ref{NG_th1}, we show the following stronger result. Under the assumptions of the theorem, there exists a positive constant $C$ that only depends on $p$ and $\theta$, such that, for any $z>0$,
\begin{equation}\label{NG_gen_result}
\Pg\left(\frac{\theta}{\theta+2}\mathcal{H}_+\geqslant\frac{\rho^2(P_n)\sigma^2}{n}z\right)\leqslant C\tau_p\left[N_0\left(1\wedge z^{-p/2}\right)+R_{P_n,p}(\mathcal{F},z)\right]
\end{equation}
where the quantity $\mathcal{H}$ is defined by
$$\mathcal{H}=\Vert s-\tilde{s}\Vert_n^2-\frac{\theta+4}{\theta}\inf_{m\in\M}\left\{\Vert s-s_m\Vert_n^2+\frac{2(\theta+2)}{\theta+4}\pen(m)\right\}$$
and we have set $R_{P_n,p}(\mathcal{F},z)$ equal to
$$\sum_{m\in\M:S_m\neq\{0\}}\left(1+\frac{\Tr(\tra{P_n}\pi_mP_n)}{\rho(\tra{P_n}\pi_mP_n)}\right)\left(\frac{L_m\Tr(\tra{P_n}\pi_mP_n)}{\rho^2(P_n)}+z\right)^{-p/2}\ .$$

Thus, for any $q>0$ such that $2(q+1)<p$, we integrate (\ref{NG_gen_result}) via Lemma \ref{lemInt} to get
\begin{eqnarray}
\E\left[\mathcal{H}_+^q\right] & = & \int_0^{\infty}qt^{q-1}\Pg\left(\mathcal{H}_+\geqslant t\right)dt\nonumber\\
& = & \left(\frac{(\theta+2)\rho^2(P_n)\sigma^2}{\theta n}\right)^q\int_0^{\infty}q z^{q-1}\Pg\left(\frac{\theta}{\theta+2}\mathcal{H}_+\geqslant\frac{\rho^2(P_n)\sigma^2}{n}z\right)dz\nonumber\\
& \leqslant & C'(p,q,\theta)\tau_p\left(\frac{\rho^2(P_n)\sigma^2}{n}\right)^qR_{P_n,\theta}^{p,q}(\mathcal{F})\label{NG_ineg04}
\end{eqnarray}
where we have set
$$R_{P_n,\theta}^{p,q}(\mathcal{F})=N_0 + \sum_{m\in\M:S_m\neq\{0\}}\left(1+\frac{\Tr(\tra{P_n}\pi_mP_n)}{\rho(\tra{P_n}\pi_mP_n)}\right)\left(\frac{L_m\Tr(\tra{P_n}\pi_mP_n)}{\rho^2(P_n)}\right)^{q-p/2}\ .$$
Since
$$\E\left[\Vert s-\tilde{s}\Vert_n^{2q}\right]^{1/q}\leqslant\E\left[\left(\frac{\theta+8}{\theta}\inf_{m\in\M}\left\{\Vert s-s_m\Vert_n^2+\frac{2(\theta+4)}{\theta+8}\pen(m)\right\}+\mathcal{H}_+\right)^{q}\right]^{1/q}\ ,$$
it follows from Minkowski's Inequality when $q\geqslant 1$ or convexity arguments when $0<q<1$ that
\begin{equation}\label{NG_ineg05}
\E\left[\Vert s-\tilde{s}\Vert_n^{2q}\right]^{1/q}\leqslant2^{(q^{-1}-1)_+}\left(C''(\theta)\inf_{m\in\M}\left\{\Vert s-s_m\Vert_n^2+\pen(m)\right\}+\E\left[\mathcal{H}_+^q\right]^{1/q}\right)\ .\end{equation}
Inequality (\ref{NG_th1_res}) directly follows from (\ref{NG_ineg04}) and (\ref{NG_ineg05}).

We now turn to the proof of (\ref{NG_gen_result}). Inequality (\ref{ineg01}) does not depend on the distribution of $\varepsilon$ and we start from here. Let $\alpha=\alpha(\theta)\in(0,1)$, for any $m\in\M$ we have
\begin{eqnarray*}
(1-\alpha)\Vert s-\tilde{s}\Vert_n^2 & \leqslant & \Vert s-s_m\Vert_n^2+\pen(m)-\pen(\hat{m})+(2-\alpha)\sigma^2\Vert\pi_{\hat{m}}P_n\varepsilon\Vert_n^2\\
& & \hspace{0.5cm}+\alpha^{-1}\sigma^2\langle u_{\hat{m}},P_n\varepsilon\rangle_n^2+2\sigma\langle s-s_m, P_n\varepsilon\rangle_n
\end{eqnarray*}
where $u_m$ is defined by (\ref{um_def}). Use again (\ref{inegCarres}) with $\alpha$ to obtain
\begin{eqnarray}
(1-\alpha)\Vert s-\tilde{s}\Vert_n^2 & \leqslant & \Vert s-s_m\Vert_n^2+\pen(m)-\pen(\hat{m})+(2-\alpha)\sigma^2\Vert\pi_{\hat{m}}P_n\varepsilon\Vert_n^2\nonumber\\
& & \hspace{0.5cm}+\alpha^{-1}\sigma^2\langle u_{\hat{m}},P_n\varepsilon\rangle_n^2+2\sigma\Vert s-s_m\Vert_n\vert\langle u_m, P_n\varepsilon\rangle_n\vert\nonumber\\
& \leqslant & (1+\alpha)\Vert s-s_m\Vert_n^2+\pen(m)-\pen(\hat{m})\label{NG_ineg01}\\
& & \hspace{0.5cm}+(2-\alpha)\sigma^2\Vert\pi_{\hat{m}}P_n\varepsilon\Vert_n^2\nonumber\\
& & \hspace{0.5cm}+\alpha^{-1}\sigma^2\langle u_{\hat{m}},P_n\varepsilon\rangle_n^2+\alpha^{-1}\sigma^2\langle u_m, P_n\varepsilon\rangle_n^2\ .\nonumber
\end{eqnarray}
Let us now introduce two functions $\bar{p}_1,\bar{p}_2:\M\rightarrow\R_+$ that will be specified later and that satisfy,
\begin{equation}\label{NG_minor_pen}
\forall m\in\M,\ \pen(m)\geqslant(2-\alpha)\bar{p}_1(m)+\alpha^{-1}\bar{p}_2(m)\ .
\end{equation}
Thus, Inequality (\ref{NG_ineg01}) implies
\begin{eqnarray*}
(1-\alpha)\Vert s-\tilde{s}\Vert_n^2 & \leqslant & (1+\alpha)\Vert s-s_m\Vert_n^2+\pen(m)+\alpha^{-1}\bar{p}_2(m)\\
& & \hspace{0.5cm}+(2-\alpha)\left(\sigma^2\Vert\pi_{\hat{m}}P_n\varepsilon\Vert_n^2-\bar{p}_1(\hat{m})\right)\\
& & \hspace{0.5cm}+\alpha^{-1}\left(\sigma^2\langle u_{\hat{m}},P_n\varepsilon\rangle_n^2-\bar{p}_2(\hat{m})\right)\\
& & \hspace{0.5cm}+\alpha^{-1}\left(\sigma^2\langle u_{m},P_n\varepsilon\rangle_n^2-\bar{p}_2(m)\right)\\
& \leqslant & (1+\alpha)\left(\Vert s-s_m\Vert_n^2+2\pen(m)/(1+\alpha)\right)\\
& & \hspace{0.5cm}+(2-\alpha)\sup_{m'\in\M}\left(\sigma^2\Vert\pi_{m'}P_n\varepsilon\Vert_n^2-\bar{p}_1(m')\right)_+\\
& & \hspace{0.5cm}+2\alpha^{-1}\sup_{m'\in\M}\left(\sigma^2\langle u_{m'},P_n\varepsilon\rangle_n^2-\bar{p}_2(m')\right)_+\ .
\end{eqnarray*}
Because the choice of $m$ is arbitrary among $\M$, we can infer that, for any $\xi>0$,
\begin{eqnarray}\Pg\left((1-\alpha)\mathcal{H}_+\geqslant\xi\right) & \leqslant & \Pg\left((2-\alpha)\sup_{m\in\M}\left(\sigma^2\Vert\pi_mP_n\varepsilon\Vert_n^2-\bar{p}_1(m)\right)_+\geqslant\frac{\xi}{2}\right)\nonumber\\
& & \hspace{0.5cm}+\Pg\left(2\alpha^{-1}\sup_{m\in\M}\left(\sigma^2\langle u_m,P_n\varepsilon\rangle_n^2-\bar{p}_2(m)\right)_+\geqslant\frac{\xi}{2}\right)\nonumber\\
& \leqslant & \sum_{m\in\M}\Pg\left(\sigma^2\Vert\pi_mP_n\varepsilon\Vert_n^2\geqslant \bar{p}_1(m)+\frac{\xi}{2(2-\alpha)}\right)\nonumber\\
& & \hspace{0.5cm} +\sum_{m\in\M}\Pg\left(\sigma^2\langle u_m,P_n\varepsilon\rangle_n^2\geqslant \bar{p}_2(m)+\frac{\alpha\xi}{4}\right)\nonumber\\
& \leqslant & \sum_{m\in\M}\Pg_{1,m}(\xi)+\sum_{m\in\M}\Pg_{2,m}(\xi)\ .\label{ineg04}
\end{eqnarray}

We first bound $\Pg_{1,m}(\xi)$. For $m\in\M$ such that $S_m=\{0\}$ (\textit{i.e.} $\pi_m=0$), $\bar{p}_1(m)\geqslant0$ leads obviously to $\Pg_{1,m}(\xi)=0$. Thus, it is sufficient to bound $\Pg_{1,m}(\xi)$ for $m$ such that $\pi_m$ is not null. This ensures that the symmetric nonnegative matrix $\tilde{A}=\tra{P_n}\pi_mP_n$ lies in $\Mn\setminus\{0\}$. Thus, under hypothesis (\ref{m_cond}), Corollary $5.1$ of \cite{Bar00} gives us, for any $x_m>0$,
$$\Pg\left(n\Vert\pi_mP_n\varepsilon\Vert_n^2\geqslant\Tr(\tilde{A})+2\sqrt{\rho(\tilde{A})\Tr(\tilde{A})x_m}+ \rho(\tilde{A})x_m\right)\leqslant \frac{C_1(p)\tau_p\Tr(\tilde{A})}{\rho(\tilde{A})x_m^{p/2}}$$
where $C_1(p)$ is a constant that only depends on $p$. The properties of the norm $\rho$ imply
\begin{equation}\label{NG_ineg02}
\rho(\tilde{A})=\rho(\tra{(\pi_mP_n)}(\pi_mP_n))=\rho(\pi_mP_n)^2\leqslant\rho^2(P_n)\ .
\end{equation}
By the inequalities (\ref{NG_ineg02}) and (\ref{inegCarres}) with $\theta/2>0$, we obtain
\begin{equation}\label{concChi2}
\Pg\left(n\Vert\pi_mP_n\varepsilon\Vert_n^2\geqslant\left(1+\frac{\theta}{2}\right)\Tr(\tilde{A})+\left(1+\frac{2}{\theta}\right)\rho^2(P_n)x_m\right)\leqslant \frac{C_1(p)\tau_p\Tr(\tilde{A})}{\rho(\tilde{A})x_m^{p/2}}\ .
\end{equation}
We take $\alpha=2/(\theta+2)\in(0,1)$ but for legibility we keep using the notation $\alpha$. Moreover, we choose
$$n\bar{p}_1(m)=\left(1+\frac{\theta}{2}+\frac{L_m}{2(\theta+1)}\right)\Tr(\tra{P_n}\pi_mP_n)\sigma^2$$
and
$$x_m=\frac{\theta}{2(\theta+1)(\theta+2)}\times\frac{L_m\Tr(\tra{P_n}\pi_mP_n)+n\xi/\sigma^2}{\rho^2(P_n)}\ .$$
Thus, Inequality (\ref{concChi2}) leads to
\begin{eqnarray}
\Pg_{1,m}(\xi) & = & \Pg\left(\sigma^2\Vert\pi_mP_n\varepsilon\Vert_n^2\geqslant \bar{p}_1(m)+\frac{\xi}{2(2-\alpha)}\right)\nonumber\\
& = & \Pg\left(\sigma^2\Vert\pi_mP_n\varepsilon\Vert_n^2\geqslant \bar{p}_1(m)+\frac{(\theta+2)\xi}{4(\theta+1)}\right)\nonumber\\
& \leqslant & \Pg\left(n\Vert\pi_mP_n\varepsilon\Vert_n^2\geqslant \left(1+\frac{\theta}{2}\right)\Tr(\tra{P_n}\pi_mP_n)+\left(1+\frac{2}{\theta}\right)\rho^2(P_n)x_m\right)\nonumber\\
& \leqslant & C_2(p,\theta)\frac{\Tr(\tra{P_n}\pi_mP_n)\tau_p}{\rho(\tra{P_n}\pi_mP_n)}\left(\frac{L_m\Tr(\tra{P_n}\pi_mP_n)+n\xi/\sigma^2}{\rho^2(P_n)}\right)^{-p/2}\ .\label{inegP1}
\end{eqnarray}

We now focus on $\Pg_{2,m}(\xi)$. Let $y_m$ be some positive real number, the Markov Inequality leads to
\begin{equation}\label{ineg03}\Pg\left(\left\vert\langle u_m,P_n\varepsilon\rangle_n\right\vert\geqslant y_m\right)\leqslant y_m^{-p}\E\left[\left\vert\langle u_m,P_n\varepsilon\rangle_n\right\vert^p\right]=y_m^{-p}\E\left[\left\vert\langle \tra{P_n}u_m,\varepsilon\rangle_n\right\vert^p\right]\ .
\end{equation}
Since $p>2$, the quantity $\tau_p$ is lower bounded by $1$,
\begin{equation}\label{minor_taup}
\tau_p=\E\left[\left\vert\varepsilon_1\right\vert^p\right]\geqslant\E\left[\varepsilon_1^2\right]^{p/2}=1\ .
\end{equation}
Moreover, we can apply the Rosenthal inequality (see Chapter 2 of \cite{Pet95}) to obtain
\begin{equation}\label{ineg02}\E\left[\left\vert\langle \tra{P_n}u_m,\varepsilon\rangle_n\right\vert^p\right]\leqslant C_3(p)n^{-p}\left(\tau_p\sum_{i=1}^n\left\vert(\tra{P_n}u_m)_i\right\vert^p+n^{p/2}\Vert \tra{P_n}u_m\Vert_n^p\right)
\end{equation}
where $C_3(p)$ is a constant that only depends on $p$. Since $p>2$, we have
$$\sum_{i=1}^n\left\vert(\tra{P_n}u_m)_i\right\vert^p\leqslant\left(\sum_{i=1}^n(\tra{P_n}u_m)_i^2\right)^{p/2}=n^{p/2}\Vert \tra{P_n}u_m\Vert_n^p\leqslant n^{p/2}\rho^p(P_n)\ .$$
Thus, the Inequality (\ref{ineg02}) becomes
$$\E\left[\left\vert\langle \tra{P_n}u_m,\varepsilon\rangle_n\right\vert^p\right]\leqslant 2C_3(p)\rho^p(P_n)\tau_pn^{-p/2}$$
and, putting this inequality in (\ref{ineg03}), we obtain
\begin{equation}\label{NG_ineg03}
\Pg\left(\left\vert\langle u_m,P_n\varepsilon\rangle_n\right\vert\geqslant y_m\right)\leqslant 2C_3(p)\rho^p(P_n)\tau_pn^{-p/2}y_m^{-p}\ .
\end{equation}
We take
$$n\bar{p}_2(m)=\frac{1}{2(\theta+1)}\sigma^2L_m\Tr(\tra{P_n}\pi_mP_n)$$
and
$$y_m^2=\frac{1}{2(\theta+2)n}\left(L_m\Tr(\tra{P_n}\pi_mP_n)+\frac{n\xi}{\sigma^2}\right)\ .$$
Finally, (\ref{NG_ineg03}) gives
\begin{eqnarray}
\Pg_{2,m}(\xi) & = & \Pg\left(\sigma^2\langle u_m,P_n\varepsilon\rangle_n^2\geqslant \bar{p}_2(m)+\frac{\alpha\xi}{4}\right)\nonumber\\
& = & \Pg\left(\sigma^2\langle u_m,P_n\varepsilon\rangle_n^2\geqslant \bar{p}_2(m)+\frac{\xi}{2(\theta+2)}\right)\nonumber\\
& \leqslant & \Pg\left(\langle u_m,P_n\varepsilon\rangle_n^2\geqslant y_m^2\right)\nonumber\\
& \leqslant & C_4(p,\theta)\tau_p\left(\frac{L_m\Tr(\tra{P_n}\pi_mP_n)+n\xi/\sigma^2}{\rho^2(P_n)}\right)^{-p/2}\ .\label{inegP2}
\end{eqnarray}
Taking
$$R'(\xi)=\sum_{m\in\M:S_m\neq\{0\}}\left(1+\frac{\Tr(\tra{P_n}\pi_mP_n)}{\rho(\tra{P_n}\pi_mP_n)}\right)\left(\frac{L_m\Tr(\tra{P_n}\pi_mP_n)+n\xi/\sigma^2}{\rho^2(P_n)}\right)^{-p/2}$$
and putting together Inequalities (\ref{ineg04}), (\ref{inegP1}) and (\ref{inegP2}) lead us to
\begin{eqnarray*}
\Pg\left((1-\alpha)\mathcal{H}_+\geqslant\xi\right) & \leqslant & \sum_{m\in\M}\Pg_{1,m}(\xi)+\sum_{m\in\M}\Pg_{2,m}(\xi)\\
& \leqslant & \sum_{m\in\M:S_m=\{0\}}\Pg_{2,m}(\xi)+\sum_{m\in\M:S_m\neq\{0\}}\Pg_{1,m}(\xi)+\sum_{m\in\M:S_m\neq\{0\}}\Pg_{2,m}(\xi)\\
& \leqslant & \sum_{m\in\M:S_m=\{0\}}1\wedge\left\{C_4(p,\theta)\tau_p\left(\frac{n\xi}{\sigma^2\rho^2(P_n)}\right)^{-p/2}\right\}+C_5(p,\theta)\tau_pR'(\xi)\\
& \leqslant & N_0(1\vee C_4(p\theta))\tau_p\left(1\wedge\left(\frac{n\xi}{\rho^2(P_n)\sigma^2}\right)^{-p/2}\right)+C_5(p,\theta)\tau_pR'(\xi)\ .
\end{eqnarray*}
For $z>0$, take $\xi=\rho^2(P_n)\sigma^2z/n$ to obtain (\ref{NG_gen_result}). We conclude the proof by computing the lowerbound (\ref{NG_minor_pen}) on the penalty function,
\begin{eqnarray*}
(2-\alpha)\bar{p}_1(m)+\alpha^{-1}\bar{p}_2(m) & = & \frac{2(\theta+1)}{\theta+2}\bar{p}_1(m)+\frac{\theta+2}{2}\bar{p}_2(m)\\
& = & \left(1+\theta+\frac{\theta^2+8\theta+8}{4(\theta+1)(\theta+2)}L_m\right)\frac{\Tr(\tra{P_n}\pi_mP_n)}{n}\sigma^2\ .\\
\end{eqnarray*}
Since $(\theta^2+8\theta+8)/(4(\theta+1)(\theta+2))\leqslant1$, the penalty given by (\ref{NG_th1_pen}) satisfies the condition (\ref{NG_minor_pen}).

\subsection{Proofs of Theorems \ref{GUV_main_theo} and \ref{NG_UV_th}}

\subsubsection{Proof of Theorem \ref{GUV_main_theo}}

Given $\theta>0$, we can find two positive numbers $\delta=\delta(\theta)<1/2$ and $\eta=\eta(\theta)$ such that $(1+\theta)(1-2\delta)\geqslant(1+2\eta)$. Thus we define
$$\Omega_n=\left\{\hat{\sigma}^2>(1-2\delta)\sigma^2\right\}\ .$$
On $\Omega_n$, we know that
$$\forall m\in\M,\ \pen(m)\geqslant(1+2\eta)\frac{Tr(\tra{P_n}\pi_mP_n)}{n}\sigma^2\ .$$
Taking care of the random nature of the penalty, we argue as in the proof of Theorem \ref{G_main_theo} with $L_m=\eta$ to get
\begin{equation}\label{ineg17}
\E\left[\Vert s-\tilde{s}\Vert_n^2\ind_{\Omega_n}\right] \leqslant \frac{\eta+1}{\eta}\inf_{m\in\M}\left\{\Vert s-s_m\Vert_n^2+\E[\pen(m)]-\frac{\Tr(\tra{P_n}\pi_mP_n)}{n}\sigma^2\right\}+\frac{\rho^2(P_n)\sigma^2}{n}R''_{P_n,\eta}(\mathcal{F})
\end{equation}
where $R''_{P_n,\eta}(\mathcal{F})$ is defined by
$$R''_{P_n,\eta}(\mathcal{F})=C_{\eta}\sum_{m\in\M}\exp\left(-\frac{C'_{\eta}\Tr(\tra{P_n}\pi_mP_n)}{\rho^2(P_n)}\right)\ .$$
We use Lemma \ref{G_UV_lem_bias} and (\ref{GUV_hyp_dim}) to get an upperbound for $\E[\pen(m)]$,
\begin{eqnarray*}
\E[\pen(m)] & \leqslant & (1+\theta)\frac{\Tr(\tra{P_n}\pi_mP_n)}{n}\sigma^2+(1+\theta)\frac{\Tr(\tra{P_n}\pi_mP_n)\Vert s-\pi s\Vert_n^2}{\Tr\left(\tra{P_n}(I_n-\pi)P_n\right)}\\
& \leqslant & (1+\theta)\frac{\Tr(\tra{P_n}\pi_mP_n)}{n}\sigma^2+(1+\theta)\frac{\Tr(\tra{P_n}P_n)\Vert s-\pi s\Vert_n^2}{\Tr\left(\tra{P_n}(I_n-\pi)P_n\right)}\\
& \leqslant & (1+\theta)\frac{\Tr(\tra{P_n}\pi_mP_n)}{n}\sigma^2+2(1+\theta)\Vert s-\pi s\Vert_n^2\ .
\end{eqnarray*}
The Proposition \ref{prop_risk} and (\ref{ineg17}) give
\begin{equation}\label{GUV_ineg_omega}
\E\left[\Vert s-\tilde{s}\Vert_n^2\ind_{\Omega_n}\right] \leqslant C(\theta)\inf_{m\in\M}\E\left[\Vert s-\hat{s}_m\Vert_n^2\right]+2(\theta+1)\Vert s-\pi s\Vert_n^2+\frac{\rho^2(P_n)\sigma^2}{n}R''_{P_n,\eta}(\mathcal{F})
\end{equation}
where $C(\theta)>1$.

We now bound $\E[\Vert s-\tilde{s}\Vert_n^2\ind_{\Omega_n^c}]$. Note that
\begin{equation*}
\Vert s-\tilde{s}\Vert_n^2 = \Vert s-s_{\hat{m}}\Vert_n^2+\sigma^2\Vert\pi_{\hat{m}}P_n\varepsilon\Vert_n^2 \leqslant \Vert s\Vert_n^2+\sigma^2\Vert P_n\varepsilon\Vert_n^2
\end{equation*}
and thus, by the Cauchy–Schwarz Inequality,
\begin{equation*}
\E[\Vert s-\tilde{s}\Vert_n^2\ind_{\Omega_n^c}] \leqslant \Vert s\Vert_n^2\Pg\left(\Omega_n^c\right)+\sigma^2\E[\Vert P_n\varepsilon\Vert_n^2\ind_{\Omega_n^c}] \leqslant \left(\Vert s\Vert_n^2+\sigma^2\E[\Vert P_n\varepsilon\Vert_n^4]^{1/2}\right)\Pg\left(\Omega_n^c\right)^{1/2}\ .
\end{equation*}
Moreover, the eigenvalues of the matrix $P_n\tra{P_n}$ are nonnegative and so
\begin{eqnarray*}
\E[\Vert P_n\varepsilon\Vert_n^4]^{1/2} & = & \left(\var(\Vert P_n\varepsilon\Vert_n^2)+E[\Vert P_n\varepsilon\Vert_n^2]^2\right)^{1/2}\\
& \leqslant & \frac{1}{n}\sqrt{\Tr(\tra{P_n}P_n)\left(\Tr(\tra{P_n}P_n)+2\rho^2(P_n)\right)}\\
& \leqslant & \frac{\Tr(\tra{P_n}P_n)+(\Tr(\tra{P_n}P_n)+2\rho^2(P_n))}{2n}\\
& \leqslant & \frac{\Tr(\tra{P_n}P_n)+\rho^2(P_n)}{n}\ .
\end{eqnarray*}
Finally, the Lemma \ref{G_UV_lem_conc} gives
\begin{eqnarray}
\E[\Vert s-\tilde{s}\Vert_n^2\ind_{\Omega_n^c}] & \leqslant & C'(\theta)\left(\Vert s\Vert_n^2+\frac{\Tr(\tra{P_n}P_n)+\rho^2(P_n)}{n}\sigma^2\right)\exp\left(-\frac{\theta^2\Tr(\tra{P_n}P_n)}{32\rho^2(P_n)}\right)\nonumber\\
& \leqslant & C'(\theta)\left(\Vert s\Vert_n^2+\frac{\rho^2(P_n)(n+1)}{n}\sigma^2\right)\exp\left(-\frac{\theta^2\Tr(\tra{P_n}P_n)}{32\rho^2(P_n)}\right)\nonumber\\
& \leqslant & C'(\theta)\left(\Vert s\Vert_n^2+2\rho^2(P_n)\sigma^2\right)\exp\left(-\frac{\theta^2\Tr(\tra{P_n}P_n)}{32\rho^2(P_n)}\right)\label{GUV_ineg_omega_comp}
\end{eqnarray}
where $C'(\theta)>1$. The inequality (\ref{GUV_main_theo_ineg}) follows by collecting (\ref{GUV_ineg_omega}) and (\ref{GUV_ineg_omega_comp}).

\subsubsection{Proof of Theorem \ref{NG_UV_th}}

Given $\theta>0$, we can find two positive numbers $\delta=\delta(\theta)<1/3$ and $\eta=\eta(\theta)$ such that $(1+\theta)(1-3\delta)\geqslant(1+2\eta)$. Thus we define
$$\Omega'_n=\left\{\hat{\sigma}^2>(1-3\delta)\sigma^2\right\}\ .$$
On $\Omega'_n$, we know that
$$\forall m\in\M,\ \pen(m)\geqslant(1+2\eta)\frac{Tr(\tra{P_n}\pi_mP_n)}{n}\sigma^2\ .$$
Let $\bar{m}$ be any element of $\M$ that minimize $\Vert s-s_{m'}\Vert_n^2+\sigma^2\Tr(\tra{P_n}\pi_{m'}P_n)/n$ among $m'\in\M$. Taking care of the random nature of the penalty, we argue as in the proof of Theorem \ref{NG_th1} with $L_m=\eta$ to get
\begin{equation*}
\E\left[\Vert s-\tilde{s}\Vert_n^{2q}\ind_{\Omega'_n}\right]^{1/q} \leqslant C(q,\theta)\E\left[\left(\Vert s-s_{\bar{m}}\Vert_n^2+\frac{\Tr(\tra{P_n}\pi_{\bar{m}}P_n)}{n}\hat{\sigma}^2\right)^q\right]^{1/q}+\frac{\rho^2(P_n)\sigma^2}{n}R_n(p,q,\theta)^{1/q}
\end{equation*}
where $R_n(p,q,\theta)$ is equal to
$$C'(p,q,\theta)\tau_p\left[N_0+\sum_{m\in\M:S\neq\{0\}}\left(1+\frac{\Tr(\tra{P_n}\pi_mP_n)}{\rho(\tra{P_n}\pi_mP_n)}\right)\left(\frac{\Tr(\tra{P_n}\pi_mP_n)}{\rho^2(P_n)}\right)^{q-p/2}\right]\ .$$
Since $q\leqslant 1$, by a convexity argument and Jensen's inequality we deduce
\begin{equation}\label{NG_ineg08}
\E\left[\Vert s-\tilde{s}\Vert_n^{2q}\ind_{\Omega'_n}\right]^{1/q} \leqslant C(q,\theta)\left(\Vert s-s_{\bar{m}}\Vert_n^2+\frac{\Tr(\tra{P_n}\pi_{\bar{m}}P_n)}{n}\E[\hat{\sigma}^2]\right)+\frac{\rho^2(P_n)\sigma^2}{n}R_n(p,q,\theta)^{1/q}\ .
\end{equation}
Lemma \ref{G_UV_lem_bias} and (\ref{GUV_hyp_dim}) give
\begin{eqnarray*}
\frac{\Tr(\tra{P_n}\pi_{\bar{m}}P_n)}{n}\E[\hat{\sigma}^2] & = & \frac{\Tr(\tra{P_n}\pi_{\bar{m}}P_n)}{n}\sigma^2+\frac{n\Tr(\tra{P_n}\pi_{\bar{m}}P_n)\Vert s-\pi s\Vert_n^2}{n\Tr(\tra{P_n}(I_n-\pi)P_n)}\\
& \leqslant & \frac{\Tr(\tra{P_n}\pi_{\bar{m}}P_n)}{n}\sigma^2+2\Vert s-\pi s\Vert_n^2\ .
\end{eqnarray*}
Thus, by the definition of $\bar{m}$ and Proposition \ref{prop_risk}, (\ref{NG_ineg08}) becomes
\begin{equation}\label{major_omega_prime}
\E\left[\Vert s-\tilde{s}\Vert_n^{2q}\ind_{\Omega'_n}\right]^{1/q} \leqslant C(q,\theta)\left(\inf_{m\in\M}\E[\Vert s-\hat{s}_m\Vert_n^2]+2\Vert s-\pi s\Vert_n^2\right)+\frac{\rho^2(P_n)\sigma^2}{n}R_n(p,q,\theta)^{1/q}\ .
\end{equation}

We now bound $\E[\Vert s-\tilde{s}\Vert_n^{2q}\ind_{{\Omega'}_n^c}]$. Note that
\begin{equation*}
\Vert s-\tilde{s}\Vert_n^2 = \Vert s-s_{\hat{m}}\Vert_n^2+\sigma^2\Vert\pi_{\hat{m}}P_n\varepsilon\Vert_n^2 \leqslant \Vert s\Vert_n^2+\sigma^2\Vert P_n\varepsilon\Vert_n^2\ .
\end{equation*}
Since $q\leqslant 1$, we have
$$\E[\Vert s-\tilde{s}\Vert_n^{2q}\ind_{{\Omega'}_n^c}]\leqslant \Vert s\Vert_n^{2q}\Pg({\Omega'}^c_n)+\sigma^{2q}\E[\Vert P_n\varepsilon\Vert_n^{2q}\ind_{{\Omega'}_n^c}]\ .$$
H\"older's Inequality with exponent $p/2q>1$ gives
$$\E[\Vert P_n\varepsilon\Vert_n^{2q}\ind_{{\Omega'}_n^c}]\leqslant\E[\Vert P_n\varepsilon\Vert_n^p]^{2q/p}\Pg({\Omega'}^c_n)^{1-2q/p}$$
and, since
$$\E[\Vert P_n\varepsilon\Vert_n^p]^{2q/p}\leqslant\rho^{2q}(P_n)\E[\Vert \varepsilon\Vert_n^p]^{2q/p}\leqslant\rho^{2q}(P_n)\tau_p^{2q/p}\ ,$$
we obtain by using Lemma \ref{NG_UV_lem_conc} that
\begin{eqnarray*}
\E[\Vert s-\tilde{s}\Vert_n^{2q}\ind_{{\Omega'}_n^c}] & \leqslant & (\Vert s\Vert_n^{2q}+\sigma^{2q}\rho^{2q}(P_n)\tau_p^{2q/p})\Pg({\Omega'}^c_n)^{1-2q/p}\\
& \leqslant & C(p,q,\theta)\kappa'_n(p,q,\theta)(\Vert s\Vert_n^{2q}+\sigma^{2q}\rho^{2q}(P_n)\tau_p^{2q/p})\left(\tau_p\rho^{\alpha_p}(P_n)\Tr(\tra{P_n}P_n)^{-\beta_p}\right)^{1-2q/p}
\end{eqnarray*}
where
$$\alpha_p=(p/2-1)\vee1\text{ and }\beta_p=(p/2-1)\wedge1\ .$$
Thus, we get
\begin{equation}\label{major_omega_primec}
\E[\Vert s-\tilde{s}\Vert_n^{2q}\ind_{{\Omega'}_n^c}]^{1/q}\leqslant C'(p,q,\theta)\kappa_n(p,q,\theta)\tau_p^{1/q}(\Vert s\Vert_n^2+\tau_p\rho^2(P_n)\sigma^2)\left(\frac{\rho^{2\alpha_p}(P_n)}{\Tr(\tra{P_n}P_n)^{\beta_p}}\right)^{1/q-2/p}\ .
\end{equation}
The announced result follows from (\ref{major_omega_prime}) and (\ref{major_omega_primec}).

\subsection{Proofs of Corollaries and Propositions}

\subsubsection{Proof of Corollary \ref{G_gen_coro1}}

Let us begin by applying Theorem \ref{G_main_theo} with constant weights $L_m=L$,
\begin{equation}\label{cor_ineg1}
\E\left[\Vert s-\tilde{s}\Vert_n^2\right] \leqslant (1+\theta^{-1})\inf_{m\in\M}\left\{\Vert s-s_m\Vert_n^2+(\theta+L)\frac{\Tr(\tra{P_n}\pi_mP_n)}{n}\sigma^2\right\}+\frac{\rho^2(P_n)\sigma^2}{n}R_n(\theta)\ .
\end{equation}
We now upperbound the remainder term. Assumption {\bf ($\text{A}_3'$)} and bounds on $N_d$ and $L$ lead to
\begin{eqnarray*}
R_n(\theta) & \leqslant & \frac{2(1+\theta)^4}{\theta^3}\sum_{m\in\M}\exp\left(-\frac{\theta^2L}{2(1+\theta)^3}\times\frac{\Tr(\tra{P_n}\pi_mP_n)}{\rho^2(P_n)}\right)\\
& \leqslant & \frac{2(1+\theta)^4}{\theta^3}\sum_{m\in\M}\exp\left(-\frac{c\theta^2L}{2(1+\theta)^3}\dim(S_m)\right)\\
& \leqslant & \frac{2(1+\theta)^4}{\theta^3}\sum_{d\in\N}N_de^{-(A+\omega)d}\\
& \leqslant & \frac{2(1+\theta)^4}{\theta^3}\sum_{d\in\N}e^{-\omega d}\ .
\end{eqnarray*}
The last bound is clearly finite and we denote it by $R=R(\theta,\omega)$. Thus, we derive from (\ref{cor_ineg1})
$$\frac{\theta}{\theta+1}\E\left[\Vert s-\tilde{s}\Vert_n^2\right]\leqslant\inf_{m\in\M}\left\{\Vert s-s_m\Vert_n^2+\left((\theta+L)\Tr(\tra{P_n}\pi_mP_n)+R\rho^2(P_n)(\dim(S_m)\vee1)\right)\frac{\sigma^2}{n}\right\}$$
and hypothesis {\bf ($\text{A}_3'$)} gives
\begin{equation*}
\frac{\theta}{\theta+1}\E\left[\Vert s-\tilde{s}\Vert_n^2\right]\leqslant\inf_{m\in\M}\left\{\Vert s-s_m\Vert_n^2+(\theta+L+R/c)\left(\Tr(\tra{P_n}\pi_mP_n)\vee c\rho^2(P_n)\right)\frac{\sigma^2}{n}\right\}
\end{equation*}
that concludes the proof.

\subsubsection{Proof of Corollary \ref{NG_gen_coro1}}

Since $p>6$, we can take $q=1$ and apply Theorem \ref{NG_th1} with constant weights $L_m=L$ to get
\begin{equation}\label{cor_ineg2}
\E\left[\Vert s-\tilde{s}\Vert_n^2\right] \leqslant C\inf_{m\in\M}\left\{\Vert s-s_m\Vert_n^2+(1+\theta+L)\frac{\Tr(\tra{P_n}\pi_mP_n)}{n}\sigma^2\right\}+\frac{\rho^2(P_n)\sigma^2}{n}R_n(p,1,\theta)\ .
\end{equation}
To upperbound the remainder term, we use Assumption {\bf ($\text{A}_3'$)} and bounds on $N_d$ and $L$ to get
\begin{eqnarray*}
R_n(p,1,\theta) & \leqslant & C'\tau_p\left[1+\sum_{m\in\M:S_m\neq\{0\}}\left(1+\frac{\Tr(\tra{P_n}\pi_mP_n)}{\rho(\tra{P_n}\pi_mP_n)}\right)\left(\frac{L\Tr(\tra{P_n}\pi_mP_n)}{\rho^2(P_n)}\right)^{1-p/2}\right]\\
& \leqslant & C'\tau_p\left[1+\sum_{m\in\M:S_m\neq\{0\}}(1+\dim(S_m))(Lc\dim(S_m))^{1-p/2}\right]\\
& \leqslant & C'\tau_p\left[1+\frac{(c\omega)'^{1-p/2}}{A}\sum_{d>0}N_d(1+d)d^{1-p/2}\right]\\
& \leqslant & C'\tau_p\left[1+(c\omega)'^{1-p/2}\sum_{d>0}(1+d)^{p/2-2-\omega}d^{1-p/2}\right]\ .
\end{eqnarray*}
The last bound is clearly finite and we denote it by $R\tau_p=R(\theta,p,\omega,\omega',c)\tau_p$. Thus, as we did in the previous proof, we derive from (\ref{cor_ineg2}) and {\bf ($\text{A}_3'$)}
$$\frac{1}{C''}\E\left[\Vert s-\tilde{s}\Vert_n^2\right]\leqslant\inf_{m\in\M}\left\{\Vert s-s_m\Vert_n^2+(1+\theta+L+R\tau_p/c)\left(\Tr(\tra{P_n}\pi_mP_n)\vee c\rho^2(P_n)\right)\frac{\sigma^2}{n}\right\}\ .$$
Since $\tau_p\geqslant1$, the announced result follows.

\subsubsection{Proof of Proposition \ref{prop_rho}}

The design points $(x_i,y^1_i,\dots,y^K_i)$ are all assumed to be independent realizations of a random variable in $[0,1]^{K+1}$ with distribution $\nu\otimes\nu_1\otimes\dots\otimes\nu_K$. We denote by $I_k$ the unit $k\times k$ matrix and, for any $a=(a_1,\dots,a_k)'\in\R^k$, we define the usual norm
$$\vert a\vert_2=\left(\sum_{i=1}^ka_i^2\right)^{1/2}\ .$$
We also consider $\delta_n=\dim(F)\leqslant D^{(1)}_n+\dots+D^{(K)}_n+1$ and $N_n=n-D_n-\delta_n$. The quantities $\delta_n$ and $N_n$ are random and only depend on the $y^j_i$'s and not on the $x_i$'s.

The space $E$ is generated by the vectors $e^{(i)}=(\phi_i(x_1),\dots,\phi_i(x_n))'$, for $i=1,\dots,D_n$. Let $\{f^{(1)},\dots,f^{(\delta_n)}\}$ be an orthonormal basis of $F$ and $\{g^{(1)},\dots,g^{(N_n)}\}$ be an orthonormal basis of $G=(E+F)^{\perp}$. In the basis $\mathbf{b}$ of $\R^n$ given by the $e^{(i)}$'s, the $f^{(i)}$'s and the $g^{(i)}$'s, the projection $P_n$ onto $E$ along $F+G$ can be expressed as
$$M=\left[
\begin{array}{cc}
I_{D_n} & 0\\
0 & 0
\end{array}
\right]\in\Mn(\R)\ .$$
Considering the matrix $C$ that transforms $\mathbf{b}$ into the canonical basis, we can decompose $P_n=CMC^{-1}$. By the properties of the norm $\rho$, we get
$$\rho^2(P_n)\leqslant\rho^2(C)\rho^2(M)\rho^2(C^{-1})=\left(\frac{1}{n}\rho(\tra{C}C)\right)\left(n\rho(\tra{C^{-1}}C^{-1})\right)\ .$$
For any $\rho>1$, we deduce from the previous inequality that
\begin{equation}\label{bound_rho_1}
\Pg(\rho(P_n)>\rho)\leqslant\Pg\left(\rho\left(\frac{\tra{C}C}{n}\right)>\rho\right)+\Pg\left(\rho(n\tra{C^{-1}}C^{-1})>\rho\right)\ .
\end{equation}
Note that for any invertible matrix $A\in\Mn(\R)$ and $\lambda>1$, if $\rho(A-I_n)<1-\lambda^{-1}$, then $\rho(A^{-1})<\lambda$. Thus, Inequality (\ref{bound_rho_1}) leads to
\begin{eqnarray}
\Pg(\rho(P_n)>\rho) & \leqslant & \Pg\left(\rho\left(\frac{\tra{C}C}{n}\right)>\rho\right)+\Pg\left(\rho\left(\frac{\tra{C}C}{n}-I_n\right)>1-\rho^{-1}\right)\nonumber\\
& \leqslant & 2\Pg\left(\rho\left(\frac{\tra{C}C}{n}-I_n\right)>1-\rho^{-1}\right)\ .\label{bound_rho_2}
\end{eqnarray}
Let us denote by $\Phi$ the $D_n\times D_n$ Gram matrix associated to the vectors $e^{(1)},\dots,e^{(D_n)}$. If we define the $D_n\times\delta_n$ matrix $\Omega$ by
$$\forall 1\leqslant i\leqslant D_n,\ \forall 1\leqslant j\leqslant\delta_n,\ \Omega_{ij}=\langle e^{(i)},f^{(j)}\rangle_n\ ,$$
then we can write the following decomposition by blocks,
$$\frac{1}{n}\tra{C}C=\left[
\begin{array}{ccc}
\Phi & \Omega & 0\\
\tra{\Omega} & I_{\delta_n} & 0\\
0 & 0 & I_{N_n}
\end{array}
\right]\in\Mn(\R)\ .$$
Consequently, by the definition of $\rho(\cdot)$, we obtain
\begin{equation}\label{bound_rho_3}
\rho\left(\frac{\tra{C}C}{n}-I_n\right)\leqslant\rho(\Phi-I_{D_n})+\rho(\Omega')
\end{equation}
where we have set
$$\Omega'=\left[
\begin{array}{cc}
0 & \Omega\\
\tra{\Omega} & 0
\end{array}\right]\ .$$
Using (\ref{bound_rho_3}) in (\ref{bound_rho_2}) leads to
\begin{equation}\label{ineqa_sum}
\Pg(\rho(P_n)>\rho)\leqslant2\Pg\left(\rho(\Phi-I_{D_n})>\frac{1-\rho^{-1}}{2}\right)+2\Pg\left(\rho(\Omega')>\frac{1-\rho^{-1}}{2}\right)=2\Pg_1+2\Pg_2\ .
\end{equation}

First, we upperbound $\Pg_1$. Let $x>0$, we consider the event
$$E_x=\left\{\forall1\leqslant i,j\leqslant D_n,\ \left\vert\langle e^{(i)},e^{(j)}\rangle_n-\int_0^1\phi_i(u)\phi_j(u)\nu(du)\right\vert\leqslant V_{ij}(\phi)\sqrt{2x}+B_{ij}(\phi)x\right\}\ .$$
Because $\Phi-I_{D_n}$ is symmetric, we know that, on the event $E_x$,
\begin{eqnarray*}
\rho(\Phi-I_{D_n}) & = & \sup_{a\in\R^{D_n},\ \vert a\vert_2\leqslant1}\vert\tra{a}(\Phi-I_{D_n})a\vert\\
& = & \sup_{a\in\R^{D_n},\ \vert a\vert_2\leqslant1}\left\vert\sum_{i=1}^{D_n}\sum_{j=1}^{D_n}a_ia_j\left(\langle e^{(i)},e^{(j)}\rangle_n-\int_0^1\phi_i(u)\phi_j(u)\nu(du)\right)\right\vert1\\
& \leqslant & \sup_{a\in\R^{D_n},\ \vert a\vert_2\leqslant1}\sum_{i=1}^{D_n}\sum_{j=1}^{D_n}\vert a_ia_j\vert\left(\vert V_{ij}(\phi)\vert\sqrt{2x}+\vert B_{ij}(\phi)\vert x\right)\\
& \leqslant & \sqrt{2xL_{\phi}}+xL_{\phi}\ .
\end{eqnarray*}
Thus, for any $x>0$ such that
\begin{equation}\label{cond_sur_x}
\sqrt{2xL_{\phi}}+xL_{\phi}\leqslant\frac{1-\rho^{-1}}{2}
\end{equation}
we deduce
\begin{eqnarray}
\Pg_1 & \leqslant & \Pg\left(\exists (i,j)\ :\ \left\vert\langle e^{(i)},e^{(j)}\rangle_n-\int_0^1\phi_i(u)\phi_j(u)\nu(du)\right\vert>V_{ij}(\phi)\sqrt{2x}+B_{ij}(\phi)x\right)\nonumber \\
& \leqslant & \sum_{i=1}^{D_n}\sum_{j=1}^{D_n}\Pg\left(\left\vert\langle e^{(i)},e^{(j)}\rangle_n-\int_0^1\phi_i(u)\phi_j(u)\nu(du)\right\vert>V_{ij}(\phi)\sqrt{2x}+B_{ij}(\phi)x\right)\ .\label{somme_majoration_p1}
\end{eqnarray}
The choice $x=(1-\rho^{-1})^2/(12L(\phi))$ satisfies (\ref{cond_sur_x}) and we apply Bernstein Inequality (see Lemma 8 of \cite{BirMas98}) to the terms of the sum in (\ref{somme_majoration_p1}) to obtain
\begin{equation}\label{major_p1}
\Pg_1\leqslant 2D_n^2\exp\left(-\frac{n(1-\rho^{-1})^2}{12L_{\phi}}\right)\ .
\end{equation}

It remains to upperbound the probability $\Pg_2$. Let $x>0$, we consider the event
$$E_x'=\left\{\forall 1\leqslant i\leqslant D_n,\ \forall 1\leqslant j\leqslant\delta_n,\ \left\vert\langle e^{(i)},f^{(j)}\rangle_n\right\vert\leqslant\sqrt{2x}+b_{\phi}\sqrt{n}x\right\}\ .$$
By definition of the norm $\rho(\cdot)$, we know that, on the event $E_x'$,
\begin{eqnarray*}
\rho(\Omega') & = & 2\sup_{\substack{a\in\R^{D_n},\ b\in\R^{\delta_n}\\\vert a\vert_2+\vert b\vert_2\leqslant1}}\left\vert \tra{a}\Omega b\right\vert\\
& \leqslant & 2\sup_{\substack{a\in\R^{D_n},\ b\in\R^{\delta_n}\\\vert a\vert_2\leqslant 1,\ \vert b\vert_2\leqslant1}}\left\vert\sum_{i=1}^{D_n}\sum_{j=1}^{\delta_n}a_ib_j\langle e^{(i)},f^{(j)}\rangle_n\right\vert\\
& \leqslant & 2\sup_{\substack{a\in\R^{D_n},\ b\in\R^{\delta_n}\\\vert a\vert_2\leqslant 1,\ \vert b\vert_2\leqslant1}}\sum_{i=1}^{D_n}\sum_{j=1}^{\delta_n}\vert a_ib_j\vert\left\vert\langle e^{(i)},f^{(j)}\rangle_n\right\vert\\
& \leqslant & 2\sqrt{D_n\delta_n}\left(\sqrt{2x}+b_{\phi}\sqrt{n}x\right)\ .
\end{eqnarray*}
Thus, for any $x>0$ such that
\begin{equation}\label{cond_sur_x_2}
2\sqrt{D_n\delta_n}\left(\sqrt{2x}+b_{\phi}\sqrt{n}x\right)\leqslant\frac{1-\rho^{-1}}{2}\ ,
\end{equation}
we apply Bernstein Inequality conditionally to the $y_i^j$'s to deduce
\begin{eqnarray}
\Pg_y\left(\rho(\Omega')>\frac{1-\rho^{-1}}{2}\right) & \leqslant & \Pg_y\left(\exists (i,j)\ :\ \left\vert\langle e^{(i)},f^{(j)}\rangle_n\right\vert>\sqrt{2x}+b_{\phi}\sqrt{n}x\right)\nonumber\\
& \leqslant & \sum_{i=1}^{D_n}\sum_{j=1}^{\delta_n}\Pg_y\left(\left\vert\langle e^{(i)},f^{(j)}\rangle_n\right\vert>\sqrt{2x}+b_{\phi}\sqrt{n}x\right)\nonumber\\
& \leqslant & 2D_n\delta_ne^{-nx}\leqslant2D_nD_n'e^{-nx}\label{ineqa_trans}
\end{eqnarray}
where $\Pg_y$ is the conditional probability given the $y_i^j$'s. Indeed, under $\Pg_y$ and (\ref{condi_phi}), the variables $\langle e^{(i)},f^{(j)}\rangle_n$ are centered with unit variance. The choice
$$x=\frac{(1-\rho^{-1})^2}{16\max\left\{4D_n\delta_n,b_{\phi}\sqrt{nD_n\delta_n}\right\}}$$
satisfies (\ref{cond_sur_x_2}) and (\ref{ineqa_trans}) leads to
\begin{eqnarray}
\Pg_2 & = & \E\left[\Pg_y\left(\rho(\Omega')>\frac{1-\rho^{-1}}{2}\right)\right]\nonumber\\
& \leqslant & 2D_nD_n'\E\left[\exp\left(-\frac{n(1-\rho^{-1})^2}{16\max\left\{4D_n\delta_n,b_{\phi}\sqrt{nD_n\delta_n}\right\}}\right)\right]\nonumber\\
& \leqslant & 2D_nD_n'\exp\left(-\frac{n(1-\rho^{-1})^2}{16\max\left\{4D_nD'_n,b_{\phi}\sqrt{nD_nD'_n}\right\}}\right)\ .\label{major_p2}
\end{eqnarray}
The announced result follows from (\ref{ineqa_sum}), (\ref{major_p1}) and (\ref{major_p2}).

\subsubsection{Proof of Proposition \ref{a_prop_rate}}

The collection $\mathcal{F}^{BM}$ is nested and, for any $d\in\N$, the quantity $N_d$ is bounded independently from $d$. Consequently, Condition (\ref{a_expcondi}) is satisfied in the Gaussian case and (\ref{a_polcondi}) is fulfilled under moment condition. In both cases, we are free to take $L=\theta=\eta/2$ and {\bf ($\text{A}_1$)} is true for $K=\eta$. Assumption {\bf ($\text{A}_3'$)} is fulfilled with $c=1/\rho^2$ and, since $\dim(S_m)>0$ for any $m\in\M$, we can apply Corollary \ref{G_gen_coro1} or \ref{NG_gen_coro1} according to whether {\bf ($\text{H}_{\text{Gau}}$)} or {\bf ($\text{H}_{\text{Mom}}$)} holds. Moreover, we denote by $\E_{\varepsilon}$ (\textit{resp.} $\E_d$) the expectation on $\varepsilon$ (\textit{resp.} the design points). So $\E_{\varepsilon,d}~[\cdot]=\E_{\varepsilon}~[\E_d[\cdot]]$.

We argue in the same way than in Section \ref{sect_estim_comp} and we use {\bf ($\text{A}_3$)} to get
\begin{eqnarray*}
\E_{\varepsilon,d}\left[\Vert s-\tilde{s}\Vert_n^2\right] & \leqslant & C\inf_{m\in\M}\left\{\E_d\left[\Vert s-s_m\Vert_n^2+\frac{\Tr(\tra{P_n}\pi_mP_n)}{n}\sigma^2\right]\right\}+C'(1+\rho)^2\left(\E_d[\Vert t-\pi_{F+G}t\Vert_n^2]+\frac{R}{n}\sigma^2\right)\\
& \leqslant & C\inf_{m\in\M}\left\{\E_d[\Vert s-s_m\Vert_n^2]+\frac{\dim(S_m)}{n}\rho^2\sigma^2\right\}+C'(1+\rho)^2\left(\E_d[\Vert t-\pi_{F+G}t\Vert_n^2]+\frac{R}{n}\sigma^2\right)\ .
\end{eqnarray*}
The definition of the norm $\Vert\cdot\Vert_n$ implies that, for any $f\in \Ltwo([0,1],\nu)$,
$$\E_d\left[\frac{1}{n}\sum_{i=1}^nf(x_i)^2\right]=\int_0^1f(x)^2\nu(dx)\ .$$
Since $s\in\mathcal{H}_{\alpha}(R)$, it is easy to see that this function lies in a Besov ball. Thus, we can apply Theorem 1 of \cite{BirMas00} and we get, for any $m\in\mathcal{M}$,
$$\E_d[\Vert s-s_m\Vert_n^2]\leqslant C(\alpha,R)\dim(S_m)^{-2\alpha}\ .$$
Arguing in the same way for the $t_j\in\Ltwo([0,1],\nu_j)$ and, since $F\perp G$, we obtain
\begin{eqnarray*}
\E_d[\Vert t-\pi_{F+G}t\Vert_n^2] & \leqslant & C(K)\sum_{j=1}^K\E_d[\Vert t^j-\pi_{F+G}t^j\Vert_n^2]\\
& \leqslant & C(K)\sum_{j=1}^K\E_d[\Vert t^j\Vert_n^2-\Vert\pi_Ft^j\Vert_n^2-\Vert\pi_Gt^j\Vert_n^2]\\
& \leqslant & C(K)\sum_{j=1}^K\E_d[\Vert t^j-\pi_Ft^j\Vert_n^2-\Vert t^j-\pi_{E+F}t^j\Vert_n^2]\\
& \leqslant & C(K)\sum_{j=1}^K\E_d[\Vert t^j-\pi_Ft^j\Vert_n^2]\\
& \leqslant & C(\alpha,R,K)D_n^{-2\alpha}\leqslant C(\alpha,R,K)\dim(S_m)^{-2\alpha}\ .
\end{eqnarray*}

Consequently, for any $m\in\mathcal{M}$, we obtain
$$E_{\varepsilon,d}\left[\Vert s-\tilde{s}\Vert_n^2\right]\leqslant C''\left(\dim(S_m)^{-2\alpha}+\frac{\dim(S_m)}{n}+\frac{1}{n}\right)\ .$$
Since $\alpha>\zeta_n$, we can consider some model $S_m$ in $\mathcal{F}^{BM}$ with dimension of order $n^{1/(2\alpha+1)}$ and derive that
$$E_{\varepsilon,d}\left[\Vert s-\tilde{s}\Vert_n^2\right]\leqslant C''\left(2n^{-2\alpha/(2\alpha+1)}+\frac{1}{n}\right)\leqslant C_{\alpha}n^{-2\alpha/(2\alpha+1)}\ .$$

\section{Lemmas}

This section is devoted to some technical results and their proofs.

\begin{lmm}\label{lemInt}
Let $p,q>0$ be two real numbers such that $2q<p$. For any $\theta>0$, the following inequality holds
$$\int_0^{\infty}\frac{qz^{q-1}}{(\theta+z)^{p/2}}dz\leqslant C(p,q)\theta^{q-p/2}$$
where $C(p,q)=p/(p-2q)$.
\end{lmm}

\begin{proof}
By splitting the integral around $\theta$, we get
\begin{eqnarray*}
\int_0^{\infty}\frac{qz^{q-1}}{(\theta+z)^{p/2}}dz & = & \int_0^{\theta}\frac{qz^{q-1}}{(\theta+z)^{p/2}}dz+\int_{\theta}^{\infty}\frac{qz^{q-1}}{(\theta+z)^{p/2}}dz\\
& \leqslant & \theta^{-p/2}\int_0^{\theta}qz^{q-1}dz+\int_{\theta}^{\infty}qz^{q-1-p/2}dz\\
& \leqslant & \left(1+\frac{2q}{p-2q}\right)\theta^{q-p/2}\ .
\end{eqnarray*}
\end{proof}

The next lemma is a variant of a lemma due to Laurent and Massart.
\begin{lmm}\label{LM}
Let $A\in\Mn\setminus\{0\}$ and $\varepsilon=(\varepsilon_1,\dots,\varepsilon_n)'$ be a standard Gaussian vector of $\R^n$. For any $x>0$, we have
\begin{equation}\label{LM_minor}
\Pg\left(n\Vert A\varepsilon\Vert_n^2\geqslant\Tr(A\tra{A})+2\sqrt{\rho(A)^2\Tr(A\tra{A})x}+2\rho(A)^2x\right)\leqslant e^{-x}
\end{equation}
and
\begin{equation}\label{LM_major}
\Pg\left(n\Vert A\varepsilon\Vert_n^2\leqslant\Tr(A\tra{A})-2\sqrt{\rho(A)^2\Tr(A\tra{A})x}\right)\leqslant e^{-x}\ .
\end{equation}
\end{lmm}

\begin{proof}
It is known that $A\varepsilon$ is a centered Gaussian vector of $\R^n$ of covariance matrix given by the positive symmetric matrix $A\tra{A}$. Let us denote by $a_1,\dots,a_n\geqslant0$ the eigenvalues of the $A\tra{A}$. Thus, the distribution of $n\Vert A\varepsilon\Vert_n^2$ is the same as the one of $\sum_{i=1}^na_i\varepsilon_i^2$. We have
$$\rho(A)^2=\max_{i=1,\dots,n}\vert a_i\vert\hspace{0.5cm}\text{and}\hspace{0.5cm}\Tr(A\tra{A})=\sum_{i=1}^na_i\ .$$
Because the $a_i$'s are nonnegative,
$$\sum_{i=1}^na_i^2\leqslant\rho(A)^2\Tr(A\tra{A})$$
and we can apply the Lemma 1 of \cite{LauMas00} to obtain the announced inequalities.
\end{proof}

We now introduce some properties that are satisfied by the estimator $\hat{\sigma}^2$ defined in (\ref{GUV_def_estim}).

\begin{lmm}\label{G_UV_lem_bias}
In the Gaussian case or under moment condition, the estimator $\hat{\sigma}^2$ satisfies
$$\E\left[\hat{\sigma}^2\right]=\sigma^2+\frac{n\Vert s-\pi s\Vert_n^2}{\Tr\left(\tra{P_n}(I_n-\pi)P_n\right)}\ .$$
\end{lmm}

\begin{proof}
We have the following decomposition
\begin{equation}\label{decomp_norm}
\Vert Y-\pi Y\Vert_n^2=\Vert s-\pi s\Vert_n^2+\sigma^2\Vert (I_n-\pi)P_n\varepsilon\Vert_n^2+2\sigma\langle s-\pi s,P_n\varepsilon\rangle_n\ .
\end{equation}
The components of $\varepsilon$ are independent and centered with unit variance. Thus, taking the expectation on both side, we obtain
$$\E\left[\Vert Y-\pi Y\Vert_n^2\right]=\Vert s-\pi s\Vert_n^2+\sigma^2\frac{\Tr(\tra{P_n}(I_n-\pi)P_n)}{n}\ .$$
\end{proof}

\begin{lmm}\label{G_UV_lem_conc}
Consider the estimator $\hat{\sigma}^2$ defined in the Gaussian case. For any $0<\delta<1/2$,
$$\Pg\left(\hat{\sigma}^2\leqslant(1-2\delta)\sigma^2\right)\leqslant C_{\delta}\exp\left(-\frac{\delta^2\Tr(\tra{P_n}P_n)}{16\rho^2(P_n)}\right)$$
where $C_{\delta}>1$ only depends on $\delta$.
\end{lmm}

\begin{proof}
Let $a\in V^{\perp}$ such that $\Vert a\Vert_n^2=1$, we set
$$
u=\left\{
\begin{array}{ll}
(s-\pi s)/\Vert s-\pi s\Vert_n & \text{if }s\neq\pi s\ ,\\
a & \text{otherwise}\ .
\end{array}
\right.
$$
We have
\begin{eqnarray*}
2\sigma\vert\langle s-\pi s,P_n\varepsilon\rangle_n\vert & = & 2\sigma\vert\langle u,P_n\varepsilon\rangle_n\vert\times\Vert s-\pi s\Vert_n\\
& \leqslant & \Vert s-\pi s\Vert_n^2+\sigma^2\langle u,P_n\varepsilon\rangle_n^2
\end{eqnarray*}
and we deduce from (\ref{decomp_norm})
\begin{eqnarray}
\Vert Y-\pi Y\Vert_n^2 & \geqslant & \sigma^2\Vert (I_n-\pi)P_n\varepsilon\Vert_n^2-\sigma^2\langle u,P_n\varepsilon\rangle_n^2\nonumber\\
& = & \sigma^2\left(\Vert P_n\varepsilon\Vert_n^2-\left(\Vert\pi P_n\varepsilon\Vert_n^2+\langle u,P_n\varepsilon\rangle_n^2\right)\right)\nonumber\\
& = & \sigma^2\left(\Vert P_n\varepsilon\Vert_n^2-\Vert\pi' P_n\varepsilon\Vert_n^2\right)\label{lem_decomp_norm}
\end{eqnarray}
where $\pi'$ is the orthogonal projection onto $V\oplus\R u$. Consequently,
\begin{eqnarray}
\Pg\left(\hat{\sigma}\leqslant(1-2\delta)\sigma^2\right) & \leqslant & \Pg\left(n\Vert P_n\varepsilon\Vert_n^2-n\Vert\pi'P_n\varepsilon\Vert_n^2\leqslant(1-2\delta)\Tr(\tra{P_n}(I_n-\pi)P_n)\right)\nonumber\\
& \leqslant & \Pg\left(n\Vert P_n\varepsilon\Vert_n^2-\Tr(\tra{P_n}P_n)\leqslant-\delta\Tr(\tra{P_n}(I_n-\pi)P_n)\right)\nonumber\\
& & \hspace{1cm}+\Pg\left(n\Vert\pi'P_n\varepsilon\Vert_n^2-\Tr(\tra{P_n}\pi P_n)\geqslant\delta\Tr(\tra{P_n}(I_n-\pi)P_n)\right)\nonumber\\
& = & \Pg_1+\Pg_2\ .\label{P1_plus_P2}
\end{eqnarray}
The Inequality (\ref{LM_major}) and (\ref{GUV_hyp_dim}) give us the following upperbound for $\Pg_1$,
\begin{equation}\label{major_P1}
\Pg_1\leqslant\exp\left(-\frac{\delta^2\Tr(\tra{P_n}(I_n-\pi)P_n)^2}{4\rho^2(P_n)\Tr(\tra{P_n}P_n)}\right)\leqslant\exp\left(-\frac{\delta^2\Tr(\tra{P_n}P_n)}{16\rho^2(P_n)}\right)\ .
\end{equation}
By the properties of the norm $\rho$, we deduce that
\begin{equation}\label{major_pi_prime}
\Tr(\tra{P_n}\pi'P_n)=\Tr(\tra{P_n}\pi P_n)+\Tr(\tra{P_n}\pi_uP_n)\leqslant\Tr(\tra{P_n}\pi P_n)+\rho^2(P_n)
\end{equation}
where we have defined $\pi_u$ as the orthogonal projection onto $\R u$.
We now apply (\ref{LM_minor}) with $A=\pi'P_n$ to obtain, for any $x>0$,
\begin{eqnarray*}
& & \Pg\left(n\Vert\pi'P_n\varepsilon\Vert_n^2\geqslant(1+\delta/2)\Tr(\tra{P_n}\pi P_n)+(1+\delta/2)\rho^2(P_n)+(2+2/\delta)x\right)\nonumber\\
& & \hspace{0.5cm}\leqslant\Pg\left(n\Vert\pi'P_n\varepsilon\Vert_n^2\geqslant(1+\delta/2)\Tr(\tra{P_n}\pi'P_n)+(2+2/\delta)x\right)\nonumber\\
& & \hspace{0.5cm}\leqslant\Pg\left(n\Vert\pi'P_n\varepsilon\Vert_n^2-\Tr(\tra{P_n}\pi'P_n)\geqslant2\sqrt{\Tr(\tra{P_n\pi'P_n})x}+2x\right)\nonumber\\
& & \hspace{0.5cm}\leqslant \exp\left(-x/\rho^2(\pi'P_n)\right)\nonumber\\
& & \hspace{0.5cm}\leqslant \exp\left(-x/\rho^2(P_n)\right)\ .\nonumber
\end{eqnarray*}
Obviously, this inequality can be extended to $x\in\R$,
\begin{equation}
\Pg\left(n\Vert\pi'P_n\varepsilon\Vert_n^2\geqslant(1+\delta/2)\Tr(\tra{P_n}\pi P_n)+(1+\delta/2)\rho^2(P_n)+(2+2/\delta)x\right)\leqslant \exp\left(-\frac{x\vee0}{\rho^2(P_n)}\right)
\end{equation}
and we take
\begin{eqnarray*}
x & = & \frac{\delta}{2(\delta+1)}\left(\delta\Tr(\tra{P_n}(I_n-\pi)P_n)-\frac{\delta}{2}\Tr(\tra{P_n}\pi P_n)-\left(1+\frac{\delta}{2}\right)\rho^2(P_n)\right)\\
& = & \frac{\delta}{2(\delta+1)}\left(\delta\Tr(\tra{P_n}P_n)-\frac{3\delta}{2}\Tr(\tra{P_n}\pi P_n)-\left(1+\frac{\delta}{2}\right)\rho^2(P_n)\right)\\
& \geqslant & \frac{\delta}{2(\delta+1)}\left(\frac{\delta\Tr(\tra{P_n}P_n)}{4}-\left(1+\frac{\delta}{2}\right)\rho^2(P_n)\right)\ .
\end{eqnarray*}
Finally, we get
\begin{eqnarray}
\Pg_2 & \leqslant & \exp\left(-\frac{\delta}{2(\delta+1)\rho^2(P_n)}\left(\frac{\delta\Tr(\tra{P_n}P_n)}{4}-\left(1+\frac{\delta}{2}\right)\rho^2(P_n)\right)_+\right)\nonumber\\
& \leqslant & \exp\left(-\frac{\delta(\delta+2)}{4(\delta+1)}\left(\frac{\delta\Tr(\tra{P_n}P_n)}{2(\delta+2)\rho^2(P_n)}-1\right)_+\right)\nonumber\\
& = & \left\{\exp\left(\frac{\delta(\delta+2)}{4(\delta+1)}\right)\times\exp\left(-\frac{\delta^2\Tr(\tra{P_n}P_n)}{8(\delta+1)\rho^2(P_n)}\right)\right\}\wedge 1\ .\label{major_P2}
\end{eqnarray}
To conclude, we use (\ref{major_P1}) and (\ref{major_P2}) in (\ref{P1_plus_P2}).
\end{proof}

\begin{lmm}\label{NG_UV_lem_conc}
Consider the estimator $\hat{\sigma}^2$ defined under moment condition. For any $0<\delta<1/3$, there exists a sequence $(\kappa_{\delta,n})_{n\in\N}$ of positive numbers that tends to a positive constant $\kappa_{\delta}$ as $\Tr(\tra{P_n}P_n)/\rho^2(P_n)$ tends to infinity, such that
$$\Pg\left(\hat{\sigma}^2\leqslant(1-3\delta)\sigma^2\right)\leqslant C(p,\delta)\kappa_{\delta,n}\tau_p\rho^{(p-2)\vee2}(P_n)\Tr(\tra{P_n}P_n)^{-((p/2-1)\wedge1)}\ .$$
\end{lmm}

\begin{proof}
We define the vector $u\in V^{\perp}$ and the projection matrix $\pi'$ as we did in the proof of Lemma \ref{G_UV_lem_conc}. The lowerbound (\ref{lem_decomp_norm}) does not depend on the distribution of $\varepsilon$ and gives
\begin{equation}\label{NG_ineg06}
\Pg\left(\hat{\sigma}^2\leqslant(1-3\delta)\sigma^2\right)\leqslant\Pg\left(n\Vert P_n\varepsilon\Vert_n^2-n\Vert\pi'P_n\varepsilon\Vert_n^2\leqslant(1-3\delta)\Tr(\tra{P_n}(I_n-\pi)P_n)\right)\ .
\end{equation}
Since the matrix $\tra{P_n}P_n$ is symmetric, we have the following decomposition
\begin{eqnarray*}
n\Vert P_n\varepsilon\Vert_n^2-\Tr(\tra{P_n}P_n) & = & n\langle \tra{P_n}P_n\varepsilon,\varepsilon\rangle_n-\Tr(\tra{P_n}P_n)\\
& = & \sum_{i=1}^n\sum_{j=1}^n(\tra{P_n}P_n)_{ij}\varepsilon_i\varepsilon_j-\Tr(\tra{P_n}P_n)\\
& = & \sum_{i=1}^n(\tra{P_n}P_n)_{ii}(\varepsilon_i^2-1)+2\sum_{i=1}^n\sum_{j>i}(\tra{P_n}P_n)_{ij}\varepsilon_i\varepsilon_j\ .
\end{eqnarray*}
Thus, (\ref{NG_ineg06}) leads to
\begin{equation}\label{P1P2P3}
\Pg\left(\hat{\sigma}^2\leqslant(1-3\delta)\sigma^2\right)\leqslant\bar{\Pg}_1+\bar{\Pg}_2+\bar{\Pg}_3
\end{equation}
where we have set
$$\bar{\Pg}_1=\Pg\left(\sum_{i=1}^n(\tra{P_n}P_n)_{ii}(\varepsilon_i^2-1)\leqslant-\delta\Tr(\tra{P_n}(I_n-\pi)P_n)\right)\ ,$$
$$\bar{\Pg}_2=\Pg\left(2\sum_{i=1}^n\sum_{j>i}(\tra{P_n}P_n)_{ij}\varepsilon_i\varepsilon_j\leqslant-\delta\Tr(\tra{P_n}(I_n-\pi)P_n)\right)$$
and
$$\bar{\Pg}_3=\Pg\left(n\Vert\pi'P_n\varepsilon\Vert_n^2-\Tr(\tra{P_n}\pi P_n)\geqslant\delta\Tr(\tra{P_n}(I_n-\pi)P_n))\right)\ .$$

Note that $\bar{\Pg}_1$ concerns a sum of independent centered random variables. By Markov's inequality and (\ref{GUV_hyp_dim}), we get
\begin{eqnarray}
\bar{\Pg}_1 & \leqslant & \Pg\left(\left\vert\sum_{i=1}^n(\tra{P_n}P_n)_{ii}(\varepsilon_i^2-1)\right\vert\geqslant\delta\Tr(\tra{P_n}(I_n-\pi)P_n)\right)\nonumber\\
& \leqslant & \delta^{-p/2}\Tr(\tra{P_n}(I_n-\pi)P_n)^{-p/2}\E\left[\left\vert\sum_{i=1}^n(\tra{P_n}P_n)_{ii}(\varepsilon_i^2-1)\right\vert^{p/2}\right]\nonumber\\
& \leqslant & 2^{p/2}\delta^{-p/2}\Tr(\tra{P_n}P_n)^{-p/2}\E\left[\left\vert\sum_{i=1}^n(\tra{P_n}P_n)_{ii}(\varepsilon_i^2-1)\right\vert^{p/2}\right]\ .\label{NG_ineg07}
\end{eqnarray}
If $p\geqslant4$ then we use the Rosenthal Inequality (see Chapter 2 of \cite{Pet95}) and (\ref{minor_taup}) to obtain
\begin{equation*}
\E\left[\left\vert\sum_{i=1}^n(\tra{P_n}P_n)_{ii}(\varepsilon_i^2-1)\right\vert^{p/2}\right]\leqslant C'(p)\tau_p\left(\sum_{i=1}^n(\tra{P_n}P_n)_{ii}^{p/2}+\left(\sum_{i=1}^n(\tra{P_n}P_n)_{ii}^2\right)^{p/4}\right)\ .
\end{equation*}
Since, for any $i\in\{1,\dots,n\}$, $(\tra{P_n}P_n)_{ii}\leqslant\rho^2(P_n)$, by a convexity argument, we get
$$\E\left[\left\vert\sum_{i=1}^n(\tra{P_n}P_n)_{ii}(\varepsilon_i^2-1)\right\vert^{p/2}\right]\leqslant2C'(p)\tau_p\rho^{p/2}(P_n)\Tr(\tra{P_n}P_n)^{p/4}\ .$$
If $2<p<4$, we refer to \cite{BahEss65} for the following inequality
\begin{equation*}
\E\left[\left\vert\sum_{i=1}^n(\tra{P_n}P_n)_{ii}(\varepsilon_i^2-1)\right\vert^{p/2}\right] \leqslant 2\sum_{i=1}^n\left\vert(\tra{P_n}P_n)_{ii}(\varepsilon_i^2-1)\right\vert^{p/2} \leqslant C''(p)\tau_p\rho^{p-2}(P_n)\Tr(\tra{P_n}P_n)\ .
\end{equation*}
In both cases, (\ref{NG_ineg07}) becomes
\begin{equation}\label{NG_P1major}
\bar{\Pg}_1\leqslant C(p)\delta^{-p/2}\tau_p\rho^{p/2}(P_n)\Tr(\tra{P_n}P_n)^{-\beta}
\end{equation}
with $\beta=(p/2-1)\wedge p/4$.

Let us now bound $\bar{\Pg}_2$. By Chebyshev's inequality, we get
\begin{eqnarray*}
\bar{\Pg}_2 & \leqslant & \Pg\left(\left\vert2\sum_{i=1}^n\sum_{j>i}(\tra{P_n}P_n)_{ij}\varepsilon_i\varepsilon_j\right\vert\geqslant\delta\Tr(\tra{P_n}(I_n-\pi)P_n)\right)\\
& \leqslant & \delta^{-2}\Tr(\tra{P_n}(I_n-\pi)P_n)^{-2}\E\left[\left(2\sum_{i=1}^n\sum_{j>i}(\tra{P_n}P_n)_{ij}\varepsilon_i\varepsilon_j\right)^2\right]\\
& \leqslant & 4\delta^{-2}\Tr(\tra{P_n}P_n)^{-2}\sum_{i=1}^n\sum_{j>i}\sum_{p=1}^n\sum_{q>p}(\tra{P_n}P_n)_{ij}(\tra{P_n}P_n)_{pq}\E[\varepsilon_i\varepsilon_j\varepsilon_p\varepsilon_q]\ .
\end{eqnarray*}
Note that, by independence between the components of $\varepsilon$, the expectation in the last sum is not null if and only if $i=p$ and $j=q$ (in this case, its value is $1$). Thus, we have
\begin{eqnarray}
\bar{\Pg}_2 & \leqslant & 4\delta^{-2}\Tr(\tra{P_n}P_n)^{-2}\sum_{i=1}^n\sum_{j>i}(\tra{P_n}P_n)_{ij}^2\nonumber\\
& \leqslant & 4\delta^{-2}\Tr(\tra{P_n}P_n)^{-2}\Tr((\tra{P_n}P_n)^2)\nonumber\\
& \leqslant & 4\delta^{-2}\rho^2(P_n)\Tr(\tra{P_n}P_n)^{-1}\ .\label{NG_P2major}
\end{eqnarray}

We finally focus on $\bar{\Pg}_3$. Recalling (\ref{major_pi_prime}), we apply Corollary 5.1 of \cite{Bar00} with $\tilde{A}=\tra{P_n}\pi'P_n$ to obtain, for any $x>0$,

\begin{eqnarray*}
& & \Pg\left(n\Vert\pi'P_n\varepsilon\Vert_n^2\geqslant(1+\delta/2)\Tr(\tra{P_n}\pi P_n)+(1+\delta/2)\rho^2(P_n)+(1+2/\delta)x\right)\nonumber\\
& & \hspace{0.5cm}\leqslant\Pg\left(n\Vert\pi'P_n\varepsilon\Vert_n^2\geqslant(1+\delta/2)\Tr(\tra{P_n}\pi'P_n)+(1+2/\delta)x\right)\nonumber\\
& & \hspace{0.5cm}\leqslant\Pg\left(n\Vert\pi'P_n\varepsilon\Vert_n^2-\Tr(\tra{P_n}\pi'P_n)\geqslant2\sqrt{\Tr(\tra{P_n\pi'P_n})x}+x\right)\nonumber\\
& & \hspace{0.5cm}\leqslant C(p)\tau_p\Tr(\tra{P_n}\pi'P_n)\rho(\pi'P_n)^{p-2}x^{-p/2}\nonumber\\
& & \hspace{0.5cm}\leqslant C(p)\tau_p\Tr(\tra{P_n}P_n)\rho^{p-2}(P_n)x^{-p/2}\ .\nonumber\\
\end{eqnarray*}
Thus, for any $x\in\R$, we define
$$
\psi(x)=\left\{
\begin{array}{ll}
C(p)\tau_p\Tr(\tra{P_n}P_n)\rho^{p-2}(P_n)x^{-p/2}\wedge1 & \text{if }x>0\\
1 & \text{if }x\leqslant0
\end{array}
\right.
$$
and $\psi(x)$ is an upperbound for
$$\Pg\left(n\Vert\pi'P_n\varepsilon\Vert_n^2\geqslant(1+\delta/2)\Tr(\tra{P_n}\pi P_n)+(1+\delta/2)\rho^2(P_n)+(1+2/\delta)x\right)\ .$$
If we take
\begin{eqnarray*}
x & = & \frac{\delta}{\delta+2}\left(\delta\Tr(\tra{P_n}(I_n-\pi)P_n)-\frac{\delta}{2}\Tr(\tra{P_n}\pi P_n)-\left(1+\frac{\delta}{2}\right)\rho^2(P_n)\right)\\
& = & \frac{\delta}{\delta+2}\left(\delta\Tr(\tra{P_n}P_n)-\frac{3\delta}{2}\Tr(\tra{P_n}\pi P_n)-\left(1+\frac{\delta}{2}\right)\rho^2(P_n)\right)\\
& \geqslant & \frac{\delta}{\delta+2}\left(\frac{\delta\Tr(\tra{P_n}P_n)}{4}-\left(1+\frac{\delta}{2}\right)\rho^2(P_n)\right)
\end{eqnarray*}
then we obtain
\begin{eqnarray}
\bar{\Pg}_3 & \leqslant & C'(p,\delta)\tau_p\frac{\Tr(\tra{P_n}P_n)\rho^{p-2}(P_n)}{\left(\delta\Tr(\tra{P_n}P_n)/4-\left(1+\delta/2\right)\rho^2(P_n)\right)_+^{p/2}}\wedge1\nonumber\\
& \leqslant & C''(p,\delta)\tau_p\frac{\Tr(\tra{P_n}P_n)^{1-p/2}\rho^{p-2}(P_n)}{\left(1-2\left(1+2/\delta\right)\rho^2(P_n)/\Tr(\tra{P_n}P_n)\right)_+^{p/2}}\wedge1\label{NG_P3major}
\end{eqnarray}
To conclude, we use (\ref{NG_P1major}), (\ref{NG_P2major}) and (\ref{NG_P3major}) in (\ref{P1P2P3}).
\end{proof}

\bibliographystyle{plain}
\bibliography{biblio}

\end{document}